\documentclass[12pt]{article}
\usepackage{e-jc}
\usepackage{amsmath}
\usepackage{amsthm}
\usepackage{amsfonts}
\usepackage{amssymb}
\usepackage{array,booktabs}
\usepackage{graphicx}
\usepackage[colorlinks=true,citecolor=black,linkcolor=black,urlcolor=blue]{hyperref}

\theoremstyle{plain}
\newtheorem{theorem}{Theorem}[section]
\newtheorem{proposition}[theorem]{Proposition}
\newtheorem{lemma}[theorem]{Lemma}
\newtheorem{corollary}[theorem]{Corollary}
\newtheorem{conj}[theorem]{Conjecture}

\theoremstyle{definition}
\newtheorem{definition}[theorem]{Definition}
\newtheorem{example}[theorem]{Example}

\theoremstyle{remark}
\newtheorem{remark}[theorem]{Remark}

\DeclareFontFamily{U}{mathx}{\hyphenchar\font45}
\DeclareFontShape{U}{mathx}{m}{n}{<-> mathx10}{}
\DeclareSymbolFont{mathx}{U}{mathx}{m}{n}
\DeclareMathAccent{\widebar}{0}{mathx}{"73}

\let\originalleft\left
\let\originalright\right
\renewcommand{\left}{\mathopen{}\mathclose\bgroup\originalleft}
\renewcommand{\right}{\aftergroup\egroup\originalright}

\makeatletter
\renewcommand{\ps@plain}{%
\renewcommand{\@oddfoot}{\footsc\hfil\footrm\thepage}}
\renewcommand{\@oddfoot}{\footsc\hfil\footrm\thepage}
\makeatother

\title{\bf Cluster automorphisms
and the marked exchange graphs of skew-symmetrizable cluster algebras}

\author{John W. Lawson\thanks{The author's studies were supported by an EPSRC PhD scholarship.}\\
\small Department of Mathematical Sciences\\[-0.8ex]
\small Durham University\\[-0.8ex]
\small South Road\\[-0.8ex]
\small Durham, UK\\[-0.8ex]
\small DH1 3LE\\
\small\tt j.w.lawson@durham.ac.uk\\
}

\date{%\dateline{Jun 16, 2016}{xxx xx, 201x}\\
\small Mathematics Subject Classifications: 13F60 (primary); 16W20 (secondary)}

\begin{document}
\maketitle
\begin{abstract}
	Cluster automorphisms have been shown to have links to the mapping class
	groups of surfaces, maximal green sequences and to exchange graph
	automorphisms for skew-symmetric cluster algebras. In this paper we
	generalise these results to the skew-symmetrizable case by introducing a
	marking on the exchange graph. Many skew-symmetrizable matrices unfold to
	skew-symmetric matrices and we consider how cluster automorphisms behave under
	this unfolding with applications to coverings of orbifolds by surfaces.

	\bigskip\noindent \textbf{Keywords:} cluster algebra; quiver mutation; cluster
	automorphism; exchange graph; mapping class group.
\end{abstract}

\section{Introduction}\label{sec:intro}

Cluster algebras were introduced by Fomin and Zelevinsky in~\cite{FZ-CA1}, and
have since found applications across many types of mathematics. These are commutative
subalgebras of $\mathbb{C}(x_1,\dotsc,x_n)$ generated by rational functions
constructed using a certain combinatorial procedure starting from an initial
seed which produces that seed's mutation class.

In the same paper Fomin and Zelevinsky defined the exchange graph of a cluster
algebra to better visualise the combinatorics of the mutation class. These
graphs proved a useful tool in their classification of finite-type cluster
algebras in~\cite{FZ-CA2} where these algebras were shown to correspond to
Dynkin diagrams.

Cluster algebras were shown to be closely related to triangulations of surfaces
by Fomin, Shapiro and Thurston in~\cite{FST-tri}, where a quiver is constructed
from a given triangulation and quiver mutations correspond to flipping an edge
in the triangulation. These quivers from surfaces play an important role in the
classification of mutation-finite quivers, given by Felikson, Shapiro and
Tumarkin in~\cite{FST-mutfin}, as all such quivers are mutation-finite and there
are only 11 other exceptional mutation classes.

In a similar fashion triangulations of orbifolds with orbifold points of order 2
were shown to correspond to mutation-finite diagrams by Felikson, Shapiro and
Tumarkin in~\cite{FST-orbifolds}. Unfoldings of diagrams, introduced by
Felikson, Shapiro and Tumarkin in~\cite{FST-unfoldings}, then correspond to
coverings of the orbifold by a  triangulated surface, as shown in~\cite[Section
12]{FST-orbifolds}.

Cluster automorphisms were introduced  by Assem, Schiffler and Shramchenko
in~\cite{ASS-auts}, for cluster algebras generated from quivers, as
automorphisms of the cluster algebra taking clusters to clusters and acting as
either the identity or the opposite function on quivers. These ideas were
extended to cluster algebras generated from certain skew-symmetrizable matrices
by Chang and Zhu in~\cite{CZ-fintype}. The group of cluster automorphisms of a
cluster algebra arising from the triangulation of a surface was shown to be
isomorphic to the mapping class group of this surface by Br\"{u}stle and Qiu
in~\cite{BQ-mcg}.

In their paper on labelled seeds and global mutations~\cite{KP-auts}, King and
Pressland showed that cluster automorphisms arise naturally when mutation
classes are considered as orbits of labelled seeds under the action of a global
mutation group $M_n$. The group of cluster automorphisms is a subgroup of the
automorphisms of these mutation classes, $\operatorname{Aut}_{M_n}$, which commute with this
group action, and in fact for mutation-finite quivers these groups are
isomorphic. We use the links between automorphisms of the exchange graph and the
labelled exchange graph to prove that this group $\operatorname{Aut}_{M_n}$ is isomorphic to
the group of exchange graph automorphisms:

{%
\renewcommand{\thetheorem}{\ref{thm:mnexch}}%
\begin{theorem}
	For a labelled mutation class $\mathcal{S}^0$ with mutation class $\mathcal{S} = 
	\mathchoice{\left.\raisebox{0.5ex}{$\mathcal{S}^0$}\middle/\raisebox{-1ex}{$\operatorname{Sym}(n)$}\right.}{\left.\mathcal{S}^0\middle/\operatorname{Sym}(n)\right.}{\mathcal{S}^0/\mathcal{S}^0}{\mathcal{S}^0/\operatorname{Sym}(n)}$ and exchange graph $\mathcal{E}(\mathcal{S})$
	\[ \operatorname{Aut}_{M_n}(\mathcal{S}^0) \cong \operatorname{Aut}\mathcal{E}(\mathcal{S}). \]%
\end{theorem}
\addtocounter{theorem}{-1}%
}

Therefore for mutation-finite quivers, such as those from triangulations of a
surface, exchange graph automorphisms are cluster automorphisms.

{%
\renewcommand{\thetheorem}{\ref{cor:mfquiv-exchaut}}%
\begin{corollary}
	For a cluster algebra $\mathcal{A}$ constructed from a mutation-finite quiver with
	exchange graph $\mathcal{E}_\mathcal{A}$
	\[\operatorname{Aut}\mathcal{E}_\mathcal{A} \cong \operatorname{Aut}\mathcal{A}.\]%
\end{corollary}
\addtocounter{theorem}{-1}%
}

This result was proved in a different way by Chang and Zhu in~\cite{CZ-exch} who
also proved an extension of this to skew-symmetrizable matrices of type $B_n$
and $C_n$ for $n \geq 3$. However for other skew-symmetrizable matrices it is
not true that exchange graph automorphisms are cluster automorphisms. It can be
shown that the group of cluster automorphisms is isomorphic to a subgroup of
the group of exchange graph automorphisms but in general there exist graph
automorphisms which do not correspond to cluster automorphisms.

In order to generalise these results we introduce a marking on the exchange
graph in such a way that any automorphism which fixes these markings does in
fact correspond to a cluster automorphism.

{%
\renewcommand{\thetheorem}{\ref{thm:mexchaut}}%
\begin{theorem}
	Let $(\textbf{x},B)$ be a seed where $B$ is a mutation-finite skew-symmetrizable
	matrix with cluster algebra $\mathcal{A}$ and marked exchange graph
	$\widehat{\mathcal{E}}_\mathcal{A}$ then
	\[ \operatorname{Aut} \mathcal{A} = \operatorname{Aut} \widehat{\mathcal{E}}_\mathcal{A}. \]%
\end{theorem}
\addtocounter{theorem}{-1}%
}

Therefore the cluster automorphisms of any cluster algebra generated by
mutation-finite skew-symmetrizable matrices can be studied using just the
combinatorial properties of its marked exchange graph.

A skew-symmetrizable matrix associated to a good orbifold with order 2 orbifold
points can be unfolded to a skew-symmetric matrix associated to a surface which
covers the orbifold. In this case we show that automorphisms of the marked
exchange graph induce automorphisms of the unfolded exchange graph.

{%
\renewcommand{\thetheorem}{\ref{thm:autsintounf}}%
\begin{theorem}
	Given a skew-symmetrizable matrix $B$ which unfolds to a matrix $Q$, with
	corresponding marked exchange graphs $\widehat{\mathcal{E}}(B)$ and $\mathcal{E}(Q) = \widehat{\mathcal{E}}(Q)$,
	\[ \operatorname{Aut}\widehat{\mathcal{E}}(B) \hookrightarrow \operatorname{Aut}\mathcal{E}(Q). \]%
\end{theorem}
\addtocounter{theorem}{-1}%
}

We finish the paper with a conjecture generalising a result of Br\"{u}stle and
Qiu linking the tagged mapping class group of a surface with the cluster
automorphisms of the corresponding surface cluster algebra.
{%
\renewcommand{\thetheorem}{\ref{conj:orbmcg}}%
\begin{conj}%
	For a cluster algebra $\mathcal{A}$ arising from the triangulation of an orbifold
	$\mathcal{O}$
	\[ \operatorname{MCG}_{\bowtie}(\mathcal{O}) \cong \operatorname{Aut}^+\mathcal{A}. \]
\end{conj}
\addtocounter{theorem}{-1}%
}
The structure of the paper is as follows: Section~\ref{sec:mut} gives basic
definitions of cluster algebras and mutations while Section~\ref{sec:exch} looks
at the exchange graph of a cluster algebra and includes proofs linking graph
automorphisms and mutation class automorphisms. Section~\ref{sec:clauts} recalls
the definition of cluster automorphisms and various known results linking these
to mutation class automorphisms and exchange graph automorphisms. The section
ends by explaining how a maximal green sequence of an acyclic quiver can be used
to construct a cluster automorphism.

In Section~\ref{sec:exchauts} we introduce the marked exchange graph which
enables us to extend these results to cluster algebras from skew-symmetrizable
matrices. We show that graph automorphisms fixing the marking are in one-to-one
correspondence with cluster automorphisms.

In Section~\ref{sec:unf} we consider unfoldings of skew-symmetrizable matrices
and show how the cluster automorphisms of a skew-symmetrizable cluster algebra
induce cluster automorphisms of its unfolded cluster algebra.
Section~\ref{sec:mcg} looks at these ideas when the skew-symmetrizable cluster
algebra is constructed from an orbifold and its unfolding gives a surface
cluster algebra.

\section{Mutations}\label{sec:mut}

A \textit{skew-symmetric} matrix is a matrix $A$ such that $A^T = -A$.
A \textit{skew-symmetrizable} matrix is a matrix $B$ such that there exists some
diagonal integer matrix $D$ with positive diagonal entries for which $BD$ is a
skew-symmetric matrix. Such a matrix with the smallest entries is called the
\textit{symmetrizing matrix} of $B$.

A \textit{quiver} is an oriented graph possibly with multiple arrows between two
vertices and in this paper we always assume that it is restricted to
having no loops or 2-cycles. If $Q$ is a quiver, then its \textit{opposite}
$Q^{\textrm{op}}$ is the quiver constructed by reversing the direction of all arrows in
$Q$.

The restrictions on the definition of a quiver ensure that quivers are in
one-to-one correspondence with skew-symmetric matrices. A given skew-symmetric
matrix $B = \left( b_{i,j} \right)_{i,j \in \lbrace 1,\dotsc,n\rbrace}$ defines
a quiver with $n$ vertices and $b_{i,j}$ arrows from the $i$-th vertex to the
$j$-th vertex if $b_{i,j} > 0$.

A \textit{diagram} is a weighted oriented graph which does not have multiple
arrows between any two vertices, in addition to having no loops or 2-cycles,
where the weights on the edges are positive integers. Similarly to quivers, if
$R$ is a diagram then its \textit{opposite} $R^{\textrm{op}}$ is constructed by reversing
all arrows in $R$.

Unlike with quivers, there is no  one-to-one correspondence between diagrams and
skew-symmetrizable matrices. Given a skew-symmetrizable matrix $B = \left(
b_{i,j} \right)_{i,j \in \lbrace 1, \dotsc, n\rbrace}$ we can construct a
diagram with $n$ vertices and an arrow from the $i$-th vertex to the $j$-th
vertex with weight $- b_{i,j} b_{j,i}$ if $b_{i,j} > 0$. Usually weights of $1$
are omitted and just shown as an unweighted arrow. However a diagram only
corresponds to a matrix if the product of weights along any chordless cycle is a
perfect square and in this case may correspond to multiple matrices.

Throughout this paper we assume that all quivers and diagrams are connected.
The results can be easily extended to disconnected diagrams, however care must
be taken as different connected components could have their arrows reversed while
other components do not, so the idea of an opposite diagram is less clear.

Let $\mathbb{K} = \mathbb{C}(x_1,\dotsc,x_n)$. A \textit{cluster} is a set of
algebraically independent elements of $\mathbb{K}$, while a \textit{labelled cluster}
is a cluster with some ordering of its elements. The individual elements in a
cluster are called \textit{cluster variables}.

A \textit{labelled seed} is a pair $(\textbf{x},B)$ where $B$ is a skew-symmetrizable
matrix and $\textbf{x}$ is a labelled cluster. Each cluster variable in the cluster
can be thought of as being attached to one of the matrix rows, or equivalently
attached to one of the vertices of the corresponding quiver or diagram.  A
\textit{seed} is a class of labelled seeds which differ only by permutations.

Throughout this paper we assume that the matrix in a seed is uniquely determined
by its cluster. This has been proved for all cluster algebras of geometric type
or generated from a non-degenerate matrix by Gekhtman, Shapiro and Vainshtein
in~\cite{GSV-exch}. In this case denote the matrix for a given cluster $\textbf{x}$
by $B(\textbf{x})$.

\begin{definition}
	Given a labelled seed $u = \left( \textbf{x}, B \right)$, where $\textbf{x} = \left(\beta_1,
	\dotsc, \beta_n \right)$ and $B = \left( b_{i,j} \right)$, then the
	\textit{mutation} $\mu_k$ acts on $u$ to give $u \cdot \mu_k = \left( \textbf{x}', B'
	\right)$ where $\textbf{x}' = \left( \beta'_i,\dotsc,\beta'_n \right)$ and $B' =
	\left( b'_{i,j} \right)$ given by
	\[ \beta'_i = \begin{cases} \beta_i & \textrm{if} \  i \not = k, \\
	\frac{\prod_{b_{j,i}>0}\beta_j^{b_{j,i}} +
	\prod_{b_{j,i}<0}\beta_j^{-b_{j,i}}}{\beta_i} & \textrm{if} \  i = k, \\
\end{cases} \]%
	\[ b'_{i,j} = \begin{cases} -b_{i,j} & \textrm{if} \  i = k \  \textrm{or} \
	j = k, \\ b_{i,j} + \frac{\left| b_{i,k} \right|b_{k,j} + b_{i,k}\left| b_{k,j} \right|}{2} &
	\textrm{otherwise}. \end{cases} \]%
\end{definition}

It is sometimes convenient to consider the local mutation $\mu_{\beta,\textbf{x}}$ of a seed
$(\textbf{x},B)$ corresponding to the mutation at the vertex associated to the cluster
variable $\beta \in \textbf{x}$. These local mutations act as functions on seeds, whereas
global mutations act on labelled seeds.

Permutations act on a labelled seed $\left( \textbf{x},B\right)$, $\textbf{x} = \left(\beta_1,
\dotsc, \beta_n \right)$, $B = \left( b_{i,j} \right)$ in the expected way
taking the $i$-th vertex to the $\sigma(i)$-th vertex and the $i$-th cluster
variable to the $\sigma(i)$-th cluster variable. Therefore
$\left(\textbf{x},B\right) \cdot \sigma = \left( \textbf{x}^\sigma, B^\sigma \right)$ where
$\textbf{x}^\sigma = \left( \beta_{\sigma^{-1}(1)}, \dotsc, \beta_{\sigma^{-1}(n)} \right)$ and
$B^\sigma = \left( b^{\sigma}_{i,j} \right)$, $b^{\sigma}_{i,j} =
b_{\sigma^{-1}(i),\sigma^{-1}(j)}$.

\begin{example}
	Given a 3 vertex seed $(\textbf{x}, B)$ as in Figure~\ref{fig:permex} and permutation
	$\sigma = \left( 1 3 2 \right)$ then $\sigma$ maps the first vertex and
	cluster variable to the third, second to first and third to second. Therefore
	$\beta^\sigma_1 = \beta_2 = \beta_{\sigma^{-1}(1)}$, $\beta^\sigma_2 =
	\beta_3$ and $\beta^\sigma_3 = \beta_1$. Similarly $B^\sigma_{1,2} = 2 =
	B_{2,3} = B_{\sigma^{-1}(1),\sigma^{-1}(2)}$ and $B^\sigma_{3,2} = -3 =
	B_{1,3}$.
\end{example}

\begin{figure}[htb]%
	\centering%
	\raisebox{-.5\height}{\shortstack{%	
		\includegraphics{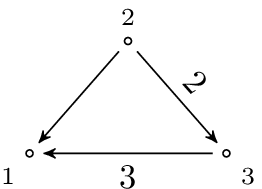} \\
		$B = \begin{pmatrix} 0 & -1 & -3 \\ 1 & 0 & 2 \\ 1 & -1 & 0 \\ 
			\end{pmatrix}$ \\
		$\textbf{x} = \left( \beta_1, \beta_2, \beta_3 \right)$
	}}
	$\qquad\xrightarrow{\quad\mbox{$\sigma = \left( 1 3 2 \right)$}\quad}\qquad$
	\raisebox{-.5\height}{\shortstack{%	
		\includegraphics{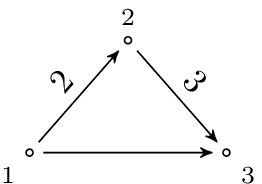} \\
		$B^{\sigma} = \begin{pmatrix} 0 & 2 & 1 \\ -1 & 0 & 1 \\ -1 & -3 & 0 \\
			\end{pmatrix}$ \\
		$\textbf{x}^{\sigma} = \left( \beta_2, \beta_3, \beta_1 \right)$
	}}
	\caption{Example of a permutation $\sigma = (1 3 2)$ acting on a seed
	$(\textbf{x},B)$ to give $(\textbf{x}, B) \cdot \sigma = (\textbf{x}^\sigma, B^\sigma)$.}
	\label{fig:permex}
\end{figure}

\begin{definition}[{\cite[Section 1]{KP-auts}}]%
	The \textit{global mutation group} for seeds of rank $n$ is given by
	\[ M_n = \left\langle \mu_1, \dotsc, \mu_n \;\middle\vert\; \mu_i^2 = 1 \right\rangle
	\rtimes \operatorname{Sym}(n) \]%
	where the $\mu_i$ are mutations and $\mu_i \sigma = \sigma \mu_{\sigma(i)}$
	for $\sigma \in \operatorname{Sym}(n)$.
\end{definition}

The \textit{labelled mutation class} $\mathcal{S}^0$ of a labelled seed $(\textbf{x}, B)$ is the
orbit of $(\textbf{x},B)$ under the action of $M_n$. The quotient by the symmetric group
action gives the \textit{mutation class} \[ \mathcal{S} = \mathchoice{\left.\raisebox{0.5ex}{$\mathcal{S}^0$}\middle/\raisebox{-1ex}{$\operatorname{Sym}(n)$}\right.}{\left.\mathcal{S}^0\middle/\operatorname{Sym}(n)\right.}{\mathcal{S}^0/\mathcal{S}^0}{\mathcal{S}^0/\operatorname{Sym}(n)}. \]%
Two seeds in the same mutation class are said to be
\textit{mutation-equivalent}.

\begin{definition}
	The \textit{cluster algebra} $\mathcal{A}(\mathcal{S})$ is the subalgebra of $\mathbb{K}$
	generated by all cluster variables occurring in the seeds in $\mathcal{S}$.
\end{definition}

A cluster algebra is said to be of \textit{finite type} if there are a finite
number of generating cluster variables in the mutation class, otherwise it is of
\textit{infinite type}. If there are a finite number of distinct matrices in
the seeds of $\mathcal{S}$, then the cluster algebra and all the matrices are said to be
\textit{mutation-finite} or of \textit{finite mutation type}, otherwise it is
\textit{mutation-infinite} or of \textit{infinite mutation type}.

\begin{definition}[{\cite[Section 2]{KP-auts}}]%
	The mutation class automorphism group
	$\operatorname{Aut}_{M_n}\left( \mathcal{S}^0 \right)$
	is the group of bijections $\phi : \mathcal{S}^0 \to \mathcal{S}^0$ which commute with the
	action of $M_n$, so for all $s \in \mathcal{S}^0$, $g \in M_n$ and $\phi \in
	\operatorname{Aut}_{M_n}(\mathcal{S}^0)$ \[ \phi(s \cdot g) = \phi(s)\cdot g. \]%
\end{definition}

\section{Exchange graphs}\label{sec:exch}

Fomin and Zelevinsky in~\cite{FZ-CA1} developed the idea of the exchange graph of a cluster
algebra to better visualise the relations in a mutation class. These were also
an important tool in their classification of finite type cluster algebras
in~\cite{FZ-CA2}.

\begin{definition}
	The \textit{exchange graph} $\mathcal{E}(\mathcal{S})$ of a mutation class $\mathcal{S}$ is constructed with
	vertices for each seed in $\mathcal{S}$ and an edge between two seeds $u$ and $v$ if and only
	if there is a single local mutation $\mu$ such that $\mu(u) = v$.

	The \textit{labelled exchange graph} $\Delta(\mathcal{S}^0)$ of a labelled mutation
	class $\mathcal{S}^0$ is constructed with a vertex for each labelled seed in $\mathcal{S}^0$ and
	an edge labelled $i$ between two labelled seeds $u$ and $v$ if and only if $u
	\cdot \mu_i = v$ (and conversely $v \cdot \mu_i  = u$).
\end{definition}

\begin{figure}
	\centering%
	\includegraphics{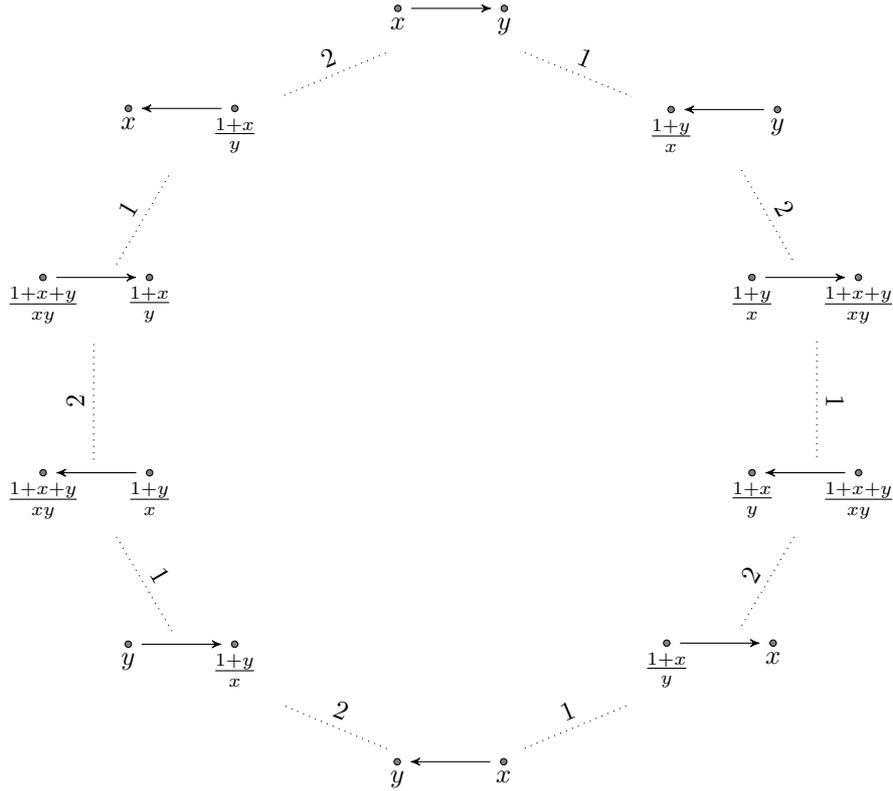}
	\caption{Labelled exchange graph for the mutation class of type $A_2$.}
	\label{fig:a2-lexch}
\end{figure}
\begin{figure}
	\centering%
	\includegraphics{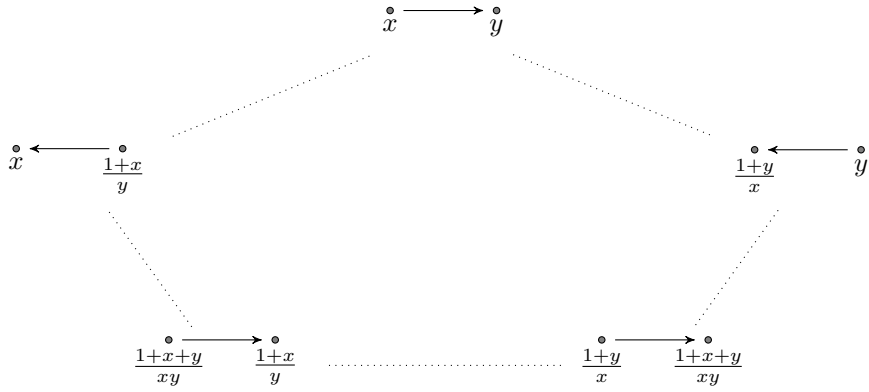}
	\caption{Exchange graph for the mutation class of type $A_2$.}
	\label{fig:a2-exch}
\end{figure}

\begin{example}[$A_2$]%
	The exchange graph for the cluster algebra of type $A_2$ is the well known
	pentagon, as seen in Figure~\ref{fig:a2-exch}. The labelled exchange graph is
	a decagon shown in Figure~\ref{fig:a2-lexch}, with the permutation acting by
	taking a seed to its antipodal seed.
\end{example}

\begin{figure}
	\centering%
	\includegraphics{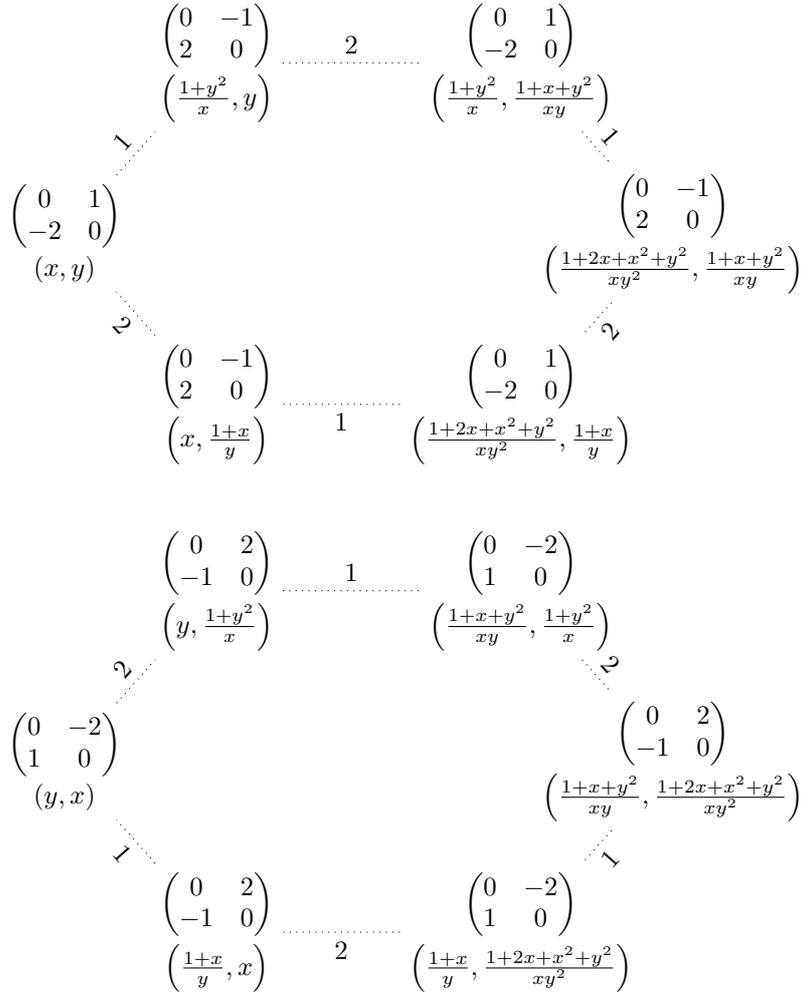}
	\caption{Labelled exchange graph for the mutation class of type $B_2$.}
	\label{fig:B2-lexch}
\end{figure}
\begin{figure}
	\centering%
	\includegraphics{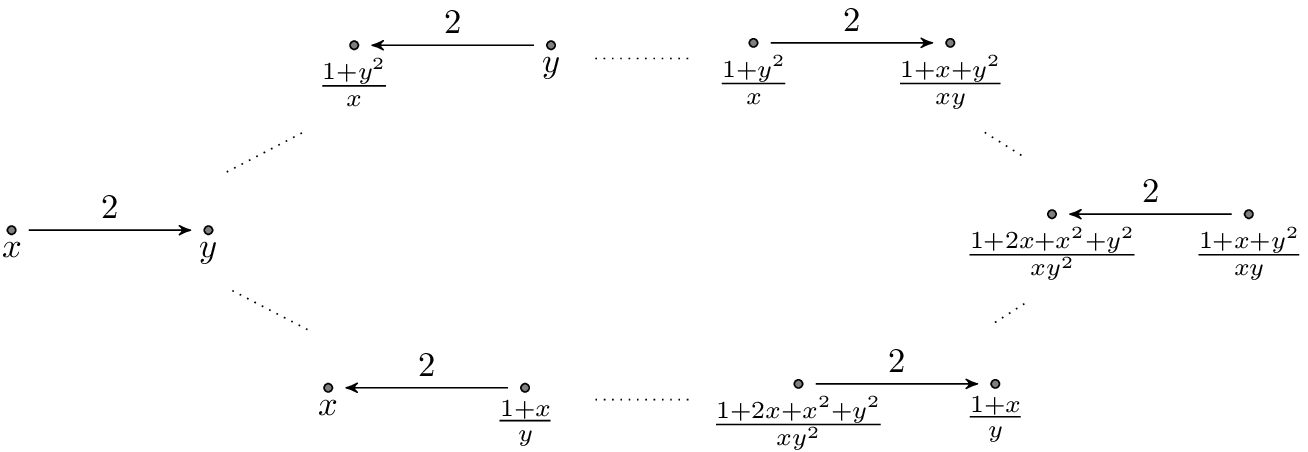}
	\caption{Exchange graph for the mutation class of type $B_2$.}
	\label{fig:B2-exch}
\end{figure}

\begin{example}[$B_2$]%
	The exchange graph for a cluster algebra of type $B_2$ is a hexagon, as shown
	in Figure~\ref{fig:B2-exch}. The labelled exchange graph however is the
	disjoint union of two hexagons as shown in Figure~\ref{fig:B2-lexch}. The
	permutation interchanging the cluster variables in a labelled seed gives
	another labelled seed which cannot be obtained from the first though just
	mutations, so any labelled seed has a permuted counterpart in the other
	connected component.
\end{example}

\begin{definition}
	The \textit{exchange graph automorphism group} $\operatorname{Aut}\mathcal{E}(\mathcal{S})$ is the group
	of permutations $\sigma$ of the vertex set of the exchange graph such that
	there is an edge between two vertices $u$ and $v$ if and only if there is an
	edge between $\sigma(u)$ and $\sigma(v)$.

	The \textit{labelled exchange graph automorphisms} in $\operatorname{Aut}\Delta(\mathcal{S}^0)$ must
	also preserve the labelling of the edges.
\end{definition}

\begin{theorem}\label{thm:pullback}
	For a labelled mutation class $\mathcal{S}^0$ with quotient $\mathcal{S}$ and corresponding
	exchange graphs $\Delta(\mathcal{S}^0)$ and $\mathcal{E}(\mathcal{S})$, then
	\[\operatorname{Aut}\mathcal{E}(\mathcal{S}) \hookrightarrow \operatorname{Aut}\Delta(\mathcal{S}^0).\]%
\end{theorem}

\begin{proof}
	To show this we construct a unique $\phi^\Delta \in \operatorname{Aut}\Delta(\mathcal{S}^0)$ for each
	$\phi \in \operatorname{Aut}\mathcal{E}(\mathcal{S})$. Let $\textbf{x}(v)$ denote the cluster of a seed $v$.

	Choose a seed $u$ in $\Delta(\mathcal{S}^0)$, then for each $i \in \lbrace 1, \dotsc, n
	\rbrace $ there is a vertex $v^i = u \cdot \mu_i$ with a corresponding edge $u
	- v^i$ labelled $i$ in the labelled exchange graph. The cluster
	$\textbf{x}(u) = (\beta_1,\dotsc,\beta_i,\dotsc,\beta_n)$ then differs from
	the cluster $\textbf{x}(v^i) = (\beta_1,\dotsc,\beta'_i,\dotsc,\beta_n)$ in just the
	$i$-th cluster variable.

	Under the quotient by the symmetric group action the labelled seed $u$ gets
	mapped to a seed $[u]$ and in the exchange graph $\mathcal{E}(\mathcal{S})$ there are edges
	$[u] - [v^i]$ for each $i \in \lbrace 1,\dotsc,n \rbrace$. Each unordered
	cluster $\textbf{x}[v^i]$ differs from the unordered cluster $\textbf{x}[u]$ in a single
	variable, just as the corresponding labelled clusters do.
	
	The exchange graph automorphism $\phi$ maps $[u]$ to some seed $\phi[u]$ and
	preserves all edges in the graph, so $\phi[u]$ is connected to $\phi[v^i]$ for each $i$.
	Therefore each $\phi[v^i]$ is a single mutation from $\phi[u]$, and so the
	unordered cluster $\textbf{x}(\phi[u]) =
	[\gamma_1,\dotsc,\gamma_{k_i},\dotsc,\gamma_n]$ differs from $\textbf{x}(\phi[v^i]) =
	[\gamma_1,\dotsc,\gamma'_{k_i},\dotsc,\gamma_n]$ in a single cluster variable.

	Set the image $\phi^\Delta(u)$ to be the seed defined by the labelled cluster
	$\textbf{x}\left(\phi^\Delta(u)\right) = \left( \gamma_{k_1}, \gamma_{k_2}, \dotsc,
	\gamma_{k_n} \right)$, obtained by choosing an order of the cluster
	$\textbf{x}(\phi[u])$ such that the $i$-th variable of
	$\textbf{x}\left(\phi^\Delta(v^i)\right)$ is the corresponding $\gamma'_{k_i}$, while
	all other variables are the same as for $\phi^\Delta(u)$.  This ensures that
	the edge between $\phi^\Delta(u)$ and $\phi^\Delta(v^i)$ is labelled $i$.
	Repeat this procedure with initial seed $v^i$ to get the ordering of the
	seeds connected to $\phi^\Delta(v^i)$.

	Continuing this construction for all seeds in $\Delta(\mathcal{S}^0)$ constructs images
	under $\phi^\Delta$ for all seeds in the labelled exchange graph.
	For any two seeds $s,t$ connected by an edge labelled $k$ in $\Delta(\mathcal{S}^0)$
	this construction ensures that the images $\phi^\Delta(s)$ and
	$\phi^\Delta(t)$ are also connected by an edge labelled $k$, and so
	$\phi^\Delta$ is indeed an automorphism of the labelled exchange graph.
\end{proof}

\begin{figure}
	\centering%
	\includegraphics{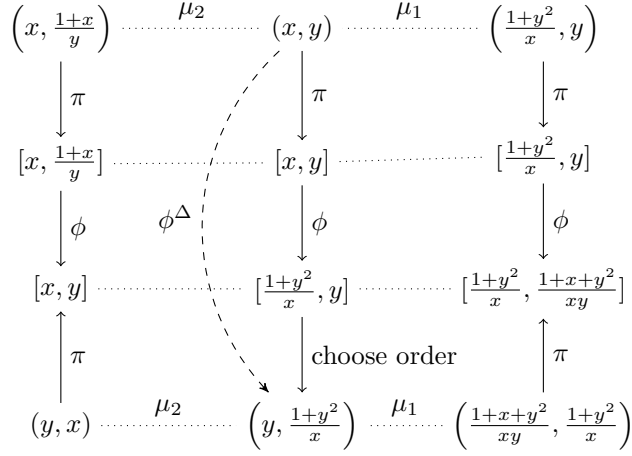}
	\caption{Commutative diagram of maps involved in Example~\ref{ex:b2-exchaut}.}
	\label{fig:b2-exchaut}
\end{figure}

\begin{example}\label{ex:b2-exchaut}
	Consider the automorphism $\phi$ of the $B_2$ exchange graph $\mathcal{E}$ shown in
	Figure~\ref{fig:B2-exch} given by a clockwise rotation by angle
	$\frac{\pi}{3}$. This automorphism pulls back to an automorphism $\phi^\Delta$
	of the labelled exchange graph $\Delta$ shown in Figure~\ref{fig:B2-lexch}.

	To determine the automorphism $\phi^\Delta$, choose an initial labelled seed
	$u = (\textbf{x},B)$ where $\textbf{x} = \left( x,y \right)$. The automorphism $\phi$ maps
	the corresponding cluster $[x,y]$ to $[\frac{1+y^2}{x},y]$ and the mutation
	$\mu_1$ takes $\left( x,y \right)$ to $\left( x,y \right) \cdot \mu_1 = \left(
	\frac{1+y^2}{x},y \right)$, whose corresponding cluster $[\frac{1+y^2}{x},y]$
	is mapped to $[\frac{1+y^2}{x},\frac{1+x+y^2}{xy}]$ by $\phi$, as shown in
	Figure~\ref{fig:b2-exchaut}.

	Denote by $\phi^\Delta \in \operatorname{Aut}\Delta$ the automorphism which
	corresponds to $\phi \in \operatorname{Aut}\mathcal{E}$ and denote the quotient by the symmetric
	group action as $\pi : \mathcal{S}^0 \to \mathcal{S}$. Then $\phi^\Delta(u)$ is a labelled seed in
	$\mathcal{S}^0$ such that $\pi\left(\phi^\Delta(u) \cdot \mu_1 \right) = \phi\left( \pi
	\left( u \cdot \mu_1 \right) \right)$.
	Hence the cluster variable which differs between $[\frac{1+y^2}{x},y]$ and
	$[\frac{1+y^2}{x},\frac{1+x+y^2}{xy}]$ needs to appear in the first position
	of the labelled cluster of $\phi^\Delta(u)$ and so
	\[\phi^\Delta\left( \textbf{x} \right) = \left(y, \frac{1+y^2}{x} \right).\]%
	This shows that the rotation of $\mathcal{E}$ actually corresponds to an
	automorphism of $\Delta$ which interchanges the two components of the graph
	(see Figure~\ref{fig:B2-lexch}) as well as rotating each component.

	Note that this automorphism takes the diagram $D =\;$
	\includegraphics{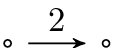} to its opposite $D^{\textrm{op}}$, however the matrix $B =
	\left(\begin{smallmatrix} 0 & 1 \\ -2 & 0\end{smallmatrix}\right)$ is not taken
	to $-B$, but rather to $-B^T$.
\end{example}

\begin{example}
	Consider the exchange graph $\mathcal{E}$ of the mutation class of type $A_2$ shown
	in Figure~\ref{fig:a2-exch}, with the labelled exchange graph $\Delta$ in
	Figure~\ref{fig:a2-lexch}. An order 5 clockwise $\frac{2\pi}{5}$ rotation
	$\phi$ of $\mathcal{E}$ is an exchange graph automorphism and so induces an
	automorphism $\phi^\Delta$ of $\Delta$.

	The cluster $\textbf{x} = [x,y]$ maps to $\phi(\textbf{x}) = [\frac{1+y}{x},y]$, so the
	labelled cluster $\hat{\textbf{x}} = \left( x,y \right)$ would be mapped to either
	$\left( \frac{1+y}{x},y\right)$ or $\left( y, \frac{1+y}{x} \right)$. To
	determine which, consider the labelled clusters adjacent to $\left( x,y
	\right)$:
	\[\left( x,y \right) \cdot \mu_1 = \left( \frac{1+y}{x}, y \right);\]%
	\[\left( x,y \right) \cdot \mu_2 = \left( x, \frac{1+x}{y} \right).\]%
	The cluster $[ \frac{1+y}{x}, y]$ is mapped to $[ \frac{1+y}{x},
	\frac{1+x+y}{xy}]$ so we need to choose an ordering for $\phi^\Delta(\hat{\textbf{x}})$
	such that $\phi^\Delta(\hat{\textbf{x}}) \cdot \mu_1$ corresponds to the same
	ordering of $[ \frac{1+y}{x}, \frac{1+x+y}{xy}]$. These two clusters
	$\phi(\textbf{x}) = [\frac{1+y}{x}, y]$ and $[\frac{1+y}{x}, \frac{1+x+y}{xy}]$
	differ by replacing $y$ with $\frac{1+x+y}{xy}$, while $\mu_1$ changes the cluster variable in the
	first position, therefore the required ordering is
	\[ \phi^\Delta(\hat{\textbf{x}}) = \left( y, \frac{1+y}{x} \right) \quad \textrm{and}
		\quad \phi^\Delta(\hat{\textbf{x}}) \cdot \mu_1 = \left( \frac{1+x+y}{xy},
	\frac{1+y}{x} \right). \]%

	This shows that $\phi$ induces the automorphism of $\Delta$ given by clockwise
	$\frac{6\pi}{5}$ rotation, which again has order 5.
\end{example}

\begin{remark}
	It is not true in general that $\operatorname{Aut}\Delta(\mathcal{S}^0) \cong \operatorname{Aut}\mathcal{E}(\mathcal{S})$, as
	$\Delta(\mathcal{S}^0)$ can have a number of connected components which are identified
	under the quotient by the symmetric group action. Any automorphism which
	changes a single connected component while fixing all others would therefore
	not project down to an automorphism of $\mathcal{E}(\mathcal{S})$. For example, in the case
	of the cluster algebra of type $B_2$, the labelled exchange graph automorphism
	given by rotating the top hexagon in Figure~\ref{fig:B2-lexch} while fixing
	the bottom hexagon would not give any valid exchange graph automorphism.
\end{remark}

Given $\phi\in \operatorname{Aut}\mathcal{E}(\mathcal{S})$ then $\phi^\Delta \in \operatorname{Aut}\Delta(\mathcal{S}^0)$ is constructed in
such a way that for $\pi : \mathcal{S}^0 \to \mathcal{S}$ the quotient by the symmetric group
action, $u\in\mathcal{S}^0$ a labelled seed and $\mu_k$ a single global mutation,
\[ \phi \left( \pi (u) \right) = \pi \left( \phi^\Delta (u) \right), \]%
\[ \phi \left( \pi \left( u \cdot \mu_k \right) \right) = \pi \left(
\phi^\Delta(u) \cdot \mu_k \right). \]%

\begin{proposition}
	The inclusion $\operatorname{Aut}\mathcal{E}(\mathcal{S}) \hookrightarrow \operatorname{Aut}\Delta(\mathcal{S}^0)$ is a homomorphism,
	that is $\left( \psi \phi \right)^\Delta = \psi^\Delta \phi^\Delta$ for any
	exchange graph automorphisms $\psi,\phi \in \operatorname{Aut}\mathcal{E}(\mathcal{S})$.
\end{proposition}

\begin{proof}
	Choose a labelled seed $u \in \mathcal{S}^0$ then
	\[ \pi\left( \left( \psi \phi \right)^\Delta (u) \right) = \left( \psi \phi
	\right) \left( \pi (u) \right) = \psi \left( \phi \left( \pi (u) \right)
\right) = \psi \left( \pi \left( \phi^\Delta (u) \right) \right) = \pi \left(
\psi^\Delta \left( \phi^\Delta (u) \right) \right). \]%
	This shows that the labelled seeds $\left( \psi \phi \right)^\Delta (u)$ and
	$\psi^\Delta \phi^\Delta (u)$ are the same up to permutation, however for any
	$k \in \lbrace 1, \dotsc, n \rbrace$
	\begin{multline*} \pi\left( \left( \psi \phi \right)^\Delta (u) \cdot \mu_k
		\right) = \left( \psi \phi \right) \left( \pi (u \cdot \mu_k) \right) = \psi
		\left( \phi \left( \pi (u \cdot \mu_k) \right) \right) = \psi \left( \pi
		\left( \phi^\Delta (u) \cdot \mu_k \right) \right) \\ = \pi \left(
	\psi^\Delta \left( \phi^\Delta (u) \cdot \mu_k \right) \right) \end{multline*}
	so after mutation in the $k$-th vertex $\left( \psi \phi \right)^\Delta (u)$
	and $\psi^\Delta \phi^\Delta (u)$ are the still same up to permutation. The
	only way that the $k$-th mutation affects two labelled seeds in the same way
	is if the labelled seeds are in fact equal and not permutations of one
	another, so
	\[ \left( \psi \phi \right)^\Delta (u) = \psi^\Delta \phi^\Delta (u) \quad
	\textrm{for any} \quad u \in \mathcal{S}^0 \]%
	and therefore $\left( \psi \phi \right)^\Delta = \psi^\Delta \phi^\Delta$.
\end{proof}

\begin{proposition}\label{prop:deltaperm}
	Let $\phi \in \operatorname{Aut}\mathcal{E}(\mathcal{S})$ with pullback $\phi^\Delta \in
	\operatorname{Aut}\Delta(\mathcal{S}^0)$, then for any labelled seed $u$ and any permutation $\sigma$
	\[ \phi^\Delta(u \cdot \sigma ) = \phi^\Delta(u) \cdot \sigma. \]%
\end{proposition}

Therefore although it looks like the construction of $\phi^\Delta$ from $\phi$
depends on the initial choice of ordering of $u$, any other ordering just gives
a permutation of $\phi^\Delta$.

\begin{proof}
	In $\Delta(\mathcal{S}^0)$ there are edges $u - v^i$ for each mutation $\mu_i$, applying
	$\sigma$ gives edges $u \cdot \sigma - v^i \cdot \sigma$ for each
	$\mu_{\sigma(i)}$. When projected $\textbf{x}(\phi[u]) = \textbf{x}(\phi[u \cdot \sigma])$
	and $\textbf{x}(\phi[v^i]) = \textbf{x}(\phi[v^i \cdot \sigma])$ for each $i$. 

	The clusters $\textbf{x}(\phi[u]) = [a, \dotsc, k_i, \dotsc]$ and $\textbf{x}(\phi[v^i]) =
	[a, \dotsc, k'_i, \dotsc]$ differ in a single cluster variable $k_i$
	to $k'_i$. In the construction of $\phi^\Delta(u)$ we specified an
	ordering $\rho_u$ on $\textbf{x}(\phi[u])$ such that the $i$-th variable
	$\rho_u(\textbf{x}(\phi[u]))_i = k_i$ for each $i$. To construct $\phi^\Delta(u \cdot
	\sigma)$ we need an ordering $\rho_{u \cdot \sigma}$ such that the
	$\sigma(i)$-th variable $\rho_{u \cdot \sigma}(\textbf{x}(\phi[u]))_{\sigma(i)} = k_i$
	so that the position of the variable which changes matches the label on the
	edge in $\Delta$. Therefore $\textbf{x}(\phi^\Delta(u \cdot \sigma)) = \rho_{u \cdot
	\sigma}(\textbf{x}(\phi[u])) = \sigma(\rho_{u}(\textbf{x}(\phi[u]))) =
	\sigma(\textbf{x}(\phi^\Delta(u))) = \textbf{x}(\phi^\Delta(u) \cdot \sigma)$.
\end{proof}

So far in this section we have proved properties of automorphisms of the
exchange graph of a cluster algebra. In the remainder of this paper we use these
results to compare these exchange graph automorphisms to other automorphisms
related to the cluster algebra.

\begin{theorem}\label{thm:mnexch}
	For a labelled mutation class $\mathcal{S}^0$ with mutation class $\mathcal{S} = 
	\mathchoice{\left.\raisebox{0.5ex}{$\mathcal{S}^0$}\middle/\raisebox{-1ex}{$\operatorname{Sym}(n)$}\right.}{\left.\mathcal{S}^0\middle/\operatorname{Sym}(n)\right.}{\mathcal{S}^0/\mathcal{S}^0}{\mathcal{S}^0/\operatorname{Sym}(n)}$ and exchange graph $\mathcal{E}(\mathcal{S})$
	\[ \operatorname{Aut}_{M_n}(\mathcal{S}^0) \cong \operatorname{Aut}\mathcal{E}(\mathcal{S}). \]%
\end{theorem}

\begin{proof}
	Let $\phi \in \operatorname{Aut}_{M_n}(\mathcal{S}^0)$ and let $\psi$ be the transformation of
	$\mathcal{E}(\mathcal{S})$ given by $\psi([u]) = [\phi(u)]$. The automorphism $\phi$
	commutes with permutations so the choice of order of $u$ does not matter,
	because for any other choice of order $u'$ there is some permutation $\sigma$
	such that $u' = u \cdot \sigma$ and then $[\phi(u')] = [\phi(u \cdot \sigma)]%
	= [\phi(u) \cdot \sigma] = [\phi(u)]$.

	For any two seeds $u$ and $v = u \cdot \mu$ related by a single mutation $\mu$
	there is an edge $[u] - [v]$ in $\mathcal{E}(\mathcal{S})$.  Then $\psi([v]) = \psi([u \cdot
	\mu]) = [\phi(u \cdot \mu)] = [\phi(u) \cdot \mu] = \tilde{\mu}[\phi(u)] =
	\tilde{\mu}\psi([u])$ where $\tilde{\mu}$ is the single local mutation on
	$[u]$ corresponding to the global mutation $\mu$ on $u$. Hence there is an
	edge $\psi[u] - \psi[v]$ in $\mathcal{E}(\mathcal{S})$, so $\psi \in \operatorname{Aut}\mathcal{E}(\mathcal{S})$ and
	$\operatorname{Aut}_{M_n}(\mathcal{S}^0) \subset \operatorname{Aut}\mathcal{E}(\mathcal{S})$.

	To show the converse, let $\psi\in\operatorname{Aut}\mathcal{E}(\mathcal{S})$ which pulls back to
	$\psi^\Delta \in \operatorname{Aut}\Delta(\mathcal{S}^0)$ by Theorem~\ref{thm:pullback}. Let $\phi
	: \mathcal{S}^0 \to \mathcal{S}^0$ be the map given by $u \mapsto \psi^\Delta(u)$.  Any
	element of $M_n$ can be written as a product of mutations and permutations, so
	to prove $\phi \in \operatorname{Aut}_{M_n}(\mathcal{S}^0)$ it suffices to show that $\phi$ commutes
	with any permutation and any mutation.

	Let $\sigma$ be a permutation, then by Proposition~\ref{prop:deltaperm} 
	\[\phi(\sigma u) = \psi^\Delta(\sigma u) = \sigma \psi^\Delta(u) = \sigma
	\phi(u). \]%

	Let $u$ and $v = u \cdot \mu$ be two labelled seeds related by a single mutation, then
	$\psi^\Delta$ is an automorphism of $\Delta(\mathcal{S}^0)$, so
	\[ \phi(u \cdot \mu) = \psi^\Delta(u \cdot \mu) = \psi^\Delta(u) \cdot \mu =
	\phi(u) \cdot \mu. \]%
\end{proof}

\section{Cluster automorphisms}\label{sec:clauts}

Cluster automorphisms were introduced by Assem, Schiffler and
Shramchenko in~\cite{ASS-auts}. In their paper the authors computed some
particular examples of automorphism groups and drew links between automorphisms
of the cluster algebra of a surface and the mapping class group of that surface.
This correspondence was later proved by Br\"{u}stle and Qiu in~\cite{BQ-mcg} for
all surfaces except a select few, as discussed in Section~\ref{sec:mcg}.

\begin{definition}[\cite{ASS-auts}]\label{def:claut}
	A $\mathbb{K}$-automorphism $f$ is a \textit{cluster automorphism} of
	$\mathcal{A}(\mathcal{S})$ if there exists a seed $(\textbf{x},B)$ in $\mathcal{S}$ such that
	\begin{enumerate}
		\item $f(\textbf{x})$ is a cluster.
		\item for every $x \in \textbf{x}$ we have $f(\mu_{x, \textbf{x}}(\textbf{x})) = \mu_{f(x),
			f(\textbf{x})}(f(\textbf{x}))$.
	\end{enumerate}
\end{definition}

Cluster automorphisms were originally only defined for skew-symmetric matrices
and hence quivers, but the same definitions and some results apply to
skew-symmetrizable matrices as well. The cluster automorphism groups in this
setting were first studied by Chang and Zhu in~\cite{CZ-exch}
and~\cite{CZ-fintype}. Recall that throughout this paper we assume that the
cluster $\textbf{x}$ of a seed uniquely determines the seed's matrix, and in this case
the matrix is denoted $B(\textbf{x})$.

\begin{lemma}[{\cite[Lemma 2.3]{ASS-auts},\cite[Lemma 2.9]{CZ-exch}}]%
\label{lem:clautmatrix}
	If $f$ is a $\mathbb{K}$-automorphism, then $f$ is a cluster automorphism if and only
	if there exists a seed $(\textbf{x}, B)$ such that $f(\textbf{x})$ is a cluster and
	$B(f(\textbf{x})) = B$ or $-B$.
\end{lemma}

The definition of a cluster automorphism only requires that there exists a
single seed such that the image is a seed and the automorphism is compatible
with mutations of that seed, however the compatibility with mutations allows
these properties to be extended to all seeds in the cluster algebra.

\begin{proposition}[{\cite[Prop 2.4]{ASS-auts}}]%
	Let $f$ be a cluster automorphism of a cluster algebra $\mathcal{A}$, then $f$
	satisfies the conditions in Definition~\ref{def:claut} and
	Lemma~\ref{lem:clautmatrix} for every seed in $\mathcal{A}$.
\end{proposition}

This therefore gives two ways of thinking of cluster automorphisms as either
automorphisms taking clusters to clusters which are compatible with mutations or
as automorphisms which fix exchange matrices (up to multiplication by -1).

\begin{definition}
	A cluster automorphism which fixes exchange matrices is called a
	\textit{direct} cluster automorphism, whereas those which send an exchange
	matrix $B$ to $-B$ are called \textit{inverse} cluster automorphisms.

	Cluster automorphisms form a group, so let
	$\operatorname{Aut}\mathcal{A}$ denote the group of all cluster
	automorphisms of $\mathcal{A}$, and
	$\operatorname{Aut}^{+}\mathcal{A}$ be the subgroup of direct
	cluster automorphisms.
\end{definition}

\begin{proposition}[{\cite[Lemma 2.9, Theorem 2.11]{ASS-auts}}]%
	Let $\mathcal{A}$ be a cluster algebra generated by an exchange matrix $B$. If $B$ is
	mutation-equivalent to $-B$ then $\operatorname{Aut}^+\mathcal{A}$ is a normal subgroup of
	$\operatorname{Aut}\mathcal{A}$ with index 2, otherwise $\operatorname{Aut}^+\mathcal{A} = \operatorname{Aut}\mathcal{A}$.
\end{proposition}

Cluster automorphisms arise naturally in the labelled seed and global mutation
setting introduced by King and Pressland, with the following correspondence:

\begin{theorem}[{\cite[Corollary 6.3]{KP-auts}}]\label{thm:mutfinquiver}
	If $\mathcal{S}$ is the mutation class of a seed $(\textbf{x},Q)$ where $Q$ is a skew-symmetric
	mutation-finite matrix then
	\[ \operatorname{Aut}_{M_n}(\mathcal{S}^0) \cong \operatorname{Aut}\mathcal{A}(\mathcal{S}). \]%
\end{theorem}

Combining Theorem~\ref{thm:mutfinquiver} with Theorem~\ref{thm:mnexch} gives the
following:

\begin{corollary}\label{cor:mfquiv-exchaut}
	For a cluster algebra $\mathcal{A}$ constructed from a mutation-finite quiver with
	exchange graph $\mathcal{E}_\mathcal{A}$
	\[\operatorname{Aut}\mathcal{E}_\mathcal{A} \cong \operatorname{Aut}\mathcal{A}.\]%
\end{corollary}

Chang and Zhu provide an alternative proof of this in~\cite{CZ-exch} and extend
the result to certain finite type skew-symmetrizable matrices:

\begin{theorem}[{\cite[Theorem 3.7]{CZ-exch}}]\label{thm:bncn}
	If $\mathcal{S}$ is the mutation class of a seed $(\textbf{x},B)$ where $B$ is a
	skew-symmetrizable matrix of Dynkin type $B_n$ or $C_n$ for $n \geq 3$ then
	\[ \operatorname{Aut}\mathcal{A}(\mathcal{S}) = \operatorname{Aut}\mathcal{E}_\mathcal{A}(\mathcal{S}). \]%
\end{theorem}

\subsection{Examples: Maximal green sequences}

Maximal green sequences are certain sequences of mutations of a given quiver.
First studied by Keller in~\cite{K-dilog} in relation to quantum dilogarithms
they have subsequently been used to study BPS states in theoretical physics
(see for example~\cite{CCC-bwm}).

In their paper on maximal green sequences, Br\"ustle, Dupont and P\'erotin
proved that any maximal green sequence for some quiver $Q$ takes it to a quiver
which is isomorphic to $Q$~\cite[Proposition 2.10]{BDP-green}. Hence this
sequence of mutations will give an element of the mutation group $\mu_{i_1}
\cdot \dotsc\cdot \mu_{i_k} \in M_n$ which takes a cluster to a cluster and a
quiver $Q$ to an isomorphic quiver $Q'$, with some permutation $\sigma$ which
acts on the vertices of $Q'$ to give $Q$. Then $\mu_{i_1} \cdot \dotsc\cdot
\mu_{i_k} \cdot \sigma \in M_n$ fixes the quiver and therefore induces a cluster
automorphism.

\begin{definition}[{\cite[Definition 2.4]{BDP-green}}]%
	Given a quiver $Q$, its \textit{framed quiver} $\hat{Q}$ (respectively
	\textit{coframed} quiver $\check{Q}$) is the quiver constructed from $Q$ by
	adding an additional vertex $\hat{i}$ and an additional arrow $i \to \hat{i}$
	(resp.\ $\hat{i} \to i$) for each vertex $i$ of $Q$.
\end{definition}

These additional vertices are considered frozen vertices of the (co)framed
quiver. For a quiver $Q$ call this set of frozen vertices of the quiver $Q^F_0$.

\begin{definition}[{\cite[Definition 2.5]{BDP-green}}]%
	Given a quiver $Q$ with framed quiver $\hat{Q}$, a non-frozen vertex $i$ of a
	quiver $R$ in the mutation class of $\hat{Q}$ is called \textit{green}
	(resp.\ \textit{red}) if for each $j \in R^F_0$ there is no arrow $j \to i$
	(resp.\ no arrow $i \to j$) in the quiver $R$.
\end{definition}

Every (non-frozen) vertex in a quiver of the mutation class of $\hat{Q}$ is
either green or red~\cite[Theorem 2.6]{BDP-green}. A maximal green sequence is
then a sequence of mutations at green vertices which continues until every
non-frozen vertex is red.

\begin{example}
The quiver of type $A_2$ has framed and coframed quivers as shown in
Figure~\ref{fig:framed-a2}. This quiver has two maximal green sequences given by
$\mu_1 \cdot \mu_2$ and $\mu_2 \cdot \mu_1 \cdot \mu_2$ which are illustrated in
Figure~\ref{fig:green-seq-a2}.

\begin{figure}
	\centering%
	\includegraphics{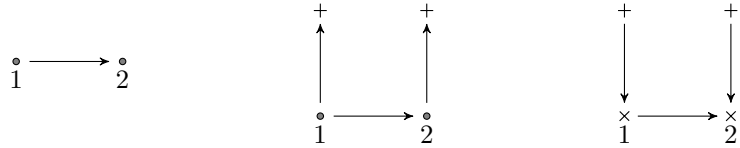}
	\caption{The quiver of type $A_2$ (left), with its framed quiver (center) and
		coframed quiver (right). Green vertices are shown as circles, red vertices
	as crosses and frozen vertices as plusses.}
	\label{fig:framed-a2}
\end{figure}
\begin{figure}
	\centering%
	\includegraphics{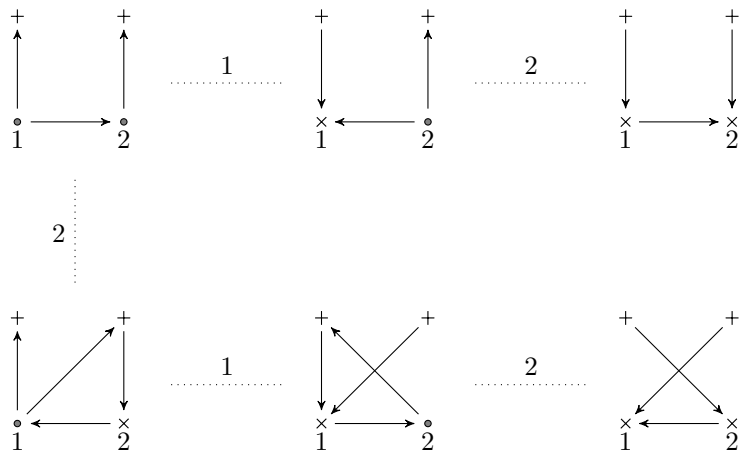}
	\caption{The two maximal green sequences of the quiver of type $A_2$ starting
		with its framed quiver. The top green sequence is $\mu_1 \cdot \mu_2$ and
		the bottom is $\mu_2 \cdot \mu_1 \cdot \mu_1$. The two resulting quivers are
		both isomorphic to the coframed quiver of the quiver of type $A_2$. Green
		vertices are shown as circles, red vertices as crosses and frozen vertices as
		plusses.}
	\label{fig:green-seq-a2}
\end{figure}

If the initial labelled seed is $(Q, \textbf{x})$ with cluster $\textbf{x} = (x, y)$, then the
resulting cluster after these green sequences induces a cluster automorphism as
shown below. The sequence $\mu_2 \cdot \mu_1 \cdot \mu_2$ does not give the same
quiver, but after the permutation $(12)$ it does:
\[ \big(Q, (x,y) \big)\cdot \mu_1 \cdot \mu_2 = \left(Q,
\left(\frac{1+y}{x},\frac{1+x+y}{xy}\right)\right) = \big(Q, (x,y) \big)\cdot
\mu_2 \cdot \mu_1 \cdot \mu_2 \cdot (12). \]%
These both give the same cluster automorphism $x \mapsto \frac{1+y}{x}$ and $y
\mapsto \frac{1+x+y}{xy}$.
\end{example}

\section{Generalising automorphisms to skew-symmetrizable case}
\label{sec:exchauts}

Theorems~\ref{thm:mutfinquiver} and~\ref{thm:bncn} show that cluster
automorphisms are linked to the automorphisms of the exchange graph for
mutation-finite skew-symmetric matrices as well as a specific family of
skew-symmetrizable matrices. However, in general the exchange graph automorphism
group for any mutation-finite skew-symmetrizable matrix is larger than the
cluster automorphism group.

An example of this would be the exchange graph automorphism of the mutation
class of $B_2$ considered in Example~\ref{ex:b2-exchaut}. This graph
automorphism does not correspond to a cluster automorphism as the initial matrix
$B$ is sent to $-B^T \neq \pm B$.

In this section we aim to generalise the results of the previous section to the
skew-symmetrizable case. To do this we introduce additional structure on the
exchange graph, which defines a marked exchange graph. This extra structure
ensures that any graph automorphism fixing this structure corresponds to a
cluster automorphism.  In this way the study of cluster automorphisms can be
reduced to the combinatorial study of graph automorphisms.

\subsection{Marked exchange graph}\label{sec:mexch}

Let $B$ be a skew-symmetrizable matrix, with symmetrizing matrix $D$. If
$\mu_i$ is any mutation, then $D$ is also the symmetrizing matrix for $B \cdot
\mu_i$. Similarly for any permutation $\sigma$ the permuted matrix $D \cdot \sigma =
\operatorname{diag}\left( d^\sigma_i \right) = \operatorname{diag}\left( d_{\sigma^{-1}(i)} \right)$ is the symmetrizing matrix for $B
\cdot \sigma = B^\sigma$.

\begin{definition}
	The \textit{marked labelled exchange graph} of a mutation class generated by
	$u = (\textbf{x},B)$ where $B$ is a skew-symmetrizable matrix with symmetrizing matrix
	$D = \operatorname{diag}\left( d_i \right)$ is the labelled exchange graph with an additional marking on
	each edge. Each edge corresponds to a global mutation $\mu_i$ for some $i$, so
	mark that edge with the symmetrizing entry $d_i$.

	If a permutation $\sigma$ acts on $u$ to give a labelled seed in a different
	component of $\Delta(\mathcal{S}^0)$, then mark the $i$-th edges with $d^\sigma_{i}$,
	where $D \cdot \sigma = \operatorname{diag}\left( d^\sigma_i \right)$.
\end{definition}

In the exchange graph $\mathcal{E}$ each edge no longer corresponds to a global
mutation $\mu_i$, but rather to a local mutation $\mu_{\beta,\textbf{x}}$ at a specific
cluster variable $\beta$ in a cluster $\textbf{x}$.

For a permutation $\sigma$ and permuted seed $(\textbf{x}, B) \cdot \sigma =
(\textbf{x}^\sigma, B^\sigma)$, then the edge $\mu_{\sigma(i)}$ adjacent to this seed
corresponds to the local mutation $\mu_{\beta^\sigma_{\sigma(i)}, [\textbf{x}^\sigma]}
= \mu_{\beta_i, [\textbf{x}]}$ as $\beta^\sigma_{\sigma(i)} =
\beta_{\sigma^{-1}(\sigma(i))} = \beta_i$ and $[\textbf{x}^\sigma] = [\textbf{x}]$. This edge
$\mu_{\sigma(i)}$ is marked with $d^\sigma_{\sigma(i)} =
d_{\sigma^{-1}(\sigma(i))} = d_i$ and hence in the quotient the edge
$\mu_{\beta_i, \textbf{x}}$ has a consistent marking, so the following is well-defined.

\begin{definition}
	Let $\widehat{\mathcal{E}}(\mathcal{S})$ be the \textit{marked exchange graph} of a mutation class
	$\mathcal{S}$ given by taking the quotient of the marked labelled exchange graph with
	respect to the symmetric group action.
	
	Alternatively let $B$ be a skew-symmetrizable matrix, with symmetrizing matrix
	$D$ and let $R$ be the diagram corresponding to $B$ so each row in $B$
	represents a vertex in $R$. Each diagonal entry in $D$ can be thought of as
	being attached to that row's vertex of $R$, and the edge in $\widehat{\mathcal{E}}$
	representing mutation in that vertex should be marked with this diagonal
	entry.
\end{definition}

\begin{figure}
	\centering%
	\includegraphics{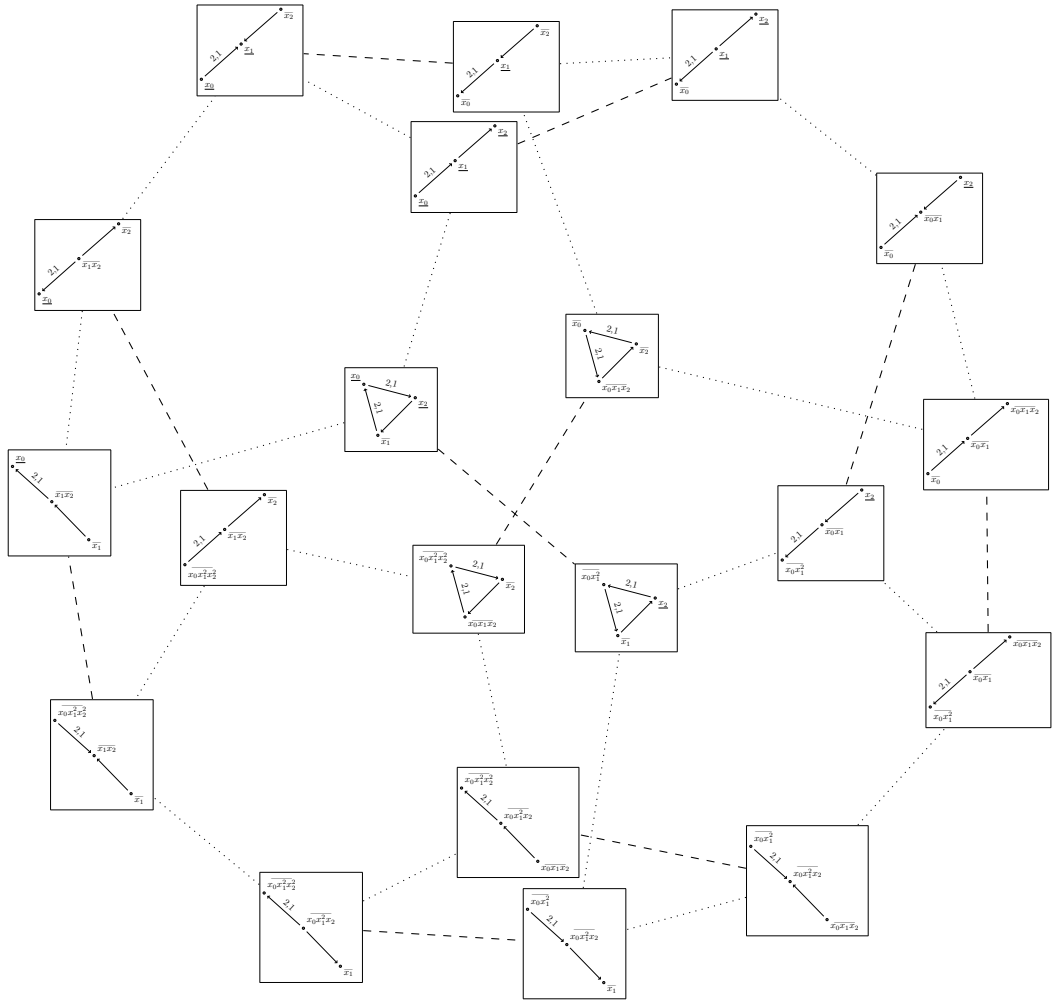}
	\caption{Marked exchange graph of type $B_3$. Dotted edges correspond to a
		symmetrizing entry of $1$, while dashed edges correspond to $2$. Only
		denominators are shown in the cluster variables with a bar above each, unless
		the cluster variable is one of $x_1, x_2$ or $x_3$ where the variable is shown
		with a bar underneath.}
	\label{fig:b3-marked}
\end{figure}

\begin{example}[$B_3$]%
	The marked exchange graph of the cluster algebra of type $B_3$ is shown in
	Figure~\ref{fig:b3-marked}. The cluster variables are not written out
	in full, rather only the denominators are shown with a bar above except for
	the initial cluster variables $x_1$, $x_2$ and $x_3$ which are shown with a bar
	underneath. Each vertex is adjacent to two dotted edges and one dashed edge.

	Choosing a matrix in the mutation class, the symmetrizing matrix is
	$\operatorname{diag}\left( 2,1,1 \right)$:
	\[\begin{pmatrix} 0 & 2 & 0 \\ -1 & 0 & 1 \\ 0 & -1 & 0
\end{pmatrix} \begin{pmatrix} 2 & 0 & 0 \\ 0 & 1 & 0 \\ 0 & 0 & 1 \end{pmatrix}
	= \begin{pmatrix} 0 & 2 & 0 \\ -2 & 0 & 1 \\ 0 & -1 & 0 \end{pmatrix}.\]%
	The dotted edges correspond to mutations in the vertices with symmetrizing
	entry $1$, while the dashed edge corresponds to the mutation in the vertex
	with symmetrizing entry $2$.

	In this case, any automorphism of the unmarked exchange graph sends dashed
	edges to dashed edges, so automatically preserves the markings and hence
	$\operatorname{Aut}\mathcal{E} = \operatorname{Aut}\widehat{\mathcal{E}}$.
\end{example}

\begin{figure}
	\centering%
	\includegraphics{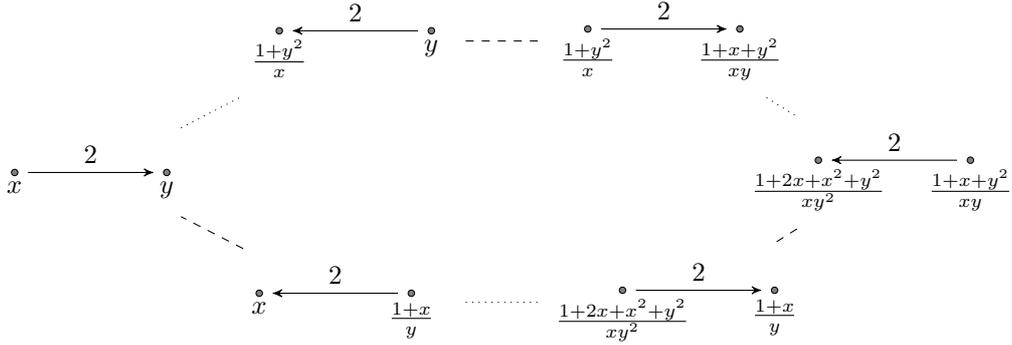}
	\caption{Marked exchange graph for cluster algebra of type $B_2$. Dotted edges
	correspond to mutations in a vertex with symmetrizer $1$ while dashed edges
	correspond to symmetrizer $2$.}
	\label{fig:b2-marked}
\end{figure}

\begin{example}[$B_2$]%
	The marked exchange graph of the cluster algebra of type $B_2$ is shown in
	Figure~\ref{fig:b2-marked}, where dotted edges correspond to mutations in
	vertices with symmetrizer $1$ and dashed edges correspond to symmetrizer $2$.
	The initial matrix for the cluster $[x,y]$ was chosen to be
	$\left(\begin{smallmatrix} 0 & 1 \\ -2 & 0 \end{smallmatrix}\right)$ with
	symmetrizing matrix $\operatorname{diag}\left( 1,2 \right)$.

	The automorphism considered in Example~\ref{ex:b2-exchaut}, given by a rotation
	of angle $\frac{\pi}{3}$, does not fix the markings in the graph, so is not an
	automorphism of the marked graph.
\end{example}

\begin{remark}
	For any skew-symmetric matrix the symmetrizing matrix is the identity, so all
	markings would be the same and $\mathcal{E}_\mathcal{A} = \widehat{\mathcal{E}}_\mathcal{A}$.
\end{remark}

\begin{remark}
	For a cluster algebra of Dynkin type $B_n$ or $C_n$, for $n \geq 3$, the
	marking on the exchange graph does not limit the number of automorphisms, so
	$\operatorname{Aut}\mathcal{E}_\mathcal{A} = \operatorname{Aut}\widehat{\mathcal{E}}_\mathcal{A}$. This follows from Theorem {3.7} in Chang
	and Zhu's paper~\cite{CZ-exch} linking exchange graph automorphisms and
	cluster automorphisms.
\end{remark}

\subsection{Geodesic loops}

\begin{definition}[{\cite[Def. 2.25]{CZ-rooted}}]%
	Let $\mathcal{E}$ be an exchange graph of a seed $u = (x,B)$ with vertices labelled
	$\left(v_i\right)_{i \in \lbrace 1,\dotsc,n\rbrace}$. For a subset of vertices
	$\lbrace v_k \rbrace$ the \textit{frozenisation} of $u$ with respect to
	$\lbrace v_k \rbrace$ is the mutation class constructed by freezing all
	vertices in $\lbrace v_k \rbrace$.

	It is often more convenient to consider the \textit{cofrozenisation} of $u$
	with respect to $\lbrace v_k \rbrace$, denoted $u \backslash \lbrace v_k
	\rbrace$, which is constructed by freezing all vertices in $u$ except those
	in $\lbrace v_k \rbrace$. This is then a frozenisation of $u$ with respect to
	$\lbrace v_i \rbrace - \lbrace v_k \rbrace$.
\end{definition}

\begin{definition}[{\cite[Section 2]{FZ-CA2}}]%
	A \textit{geodesic loop} $\mathcal{L} = \mathcal{L}^{a,b}_u$ is the exchange graph of a
	cofrozenisation $u \backslash \lbrace a,b \rbrace$ which leaves only two vertices
	$a$ and $b$ unfrozen. A loop is then either a polygon with $4,5,6$ or $8$
	sides or an infinite line, which embeds into the exchange graph of the
	mutation class of $u$.

	The \textit{distance} between a geodesic loop $\mathcal{L}$ and any vertex $v$ in
	$\mathcal{E}$ is the (possibly zero) minimum number of edges in $\mathcal{E}$ between $v$
	and any vertex in $\mathcal{L}$.

	The \textit{length} of a geodesic loop $\operatorname{Len}\left( \mathcal{L} \right) \in \lbrace
	4,5,6,8,\infty \rbrace$ is the number of edges in the loop.
\end{definition}

Geodesic loops as subgraphs of a larger exchange graph give rise to the
following sets, which encode the information about a given seed represented by a
vertex of the exchange graph. The following construction is a slight notational
variation of the one given by Chang and Zhu in Definition 3.1 of~\cite{CZ-exch}.

\begin{definition}
	Let $u$ be a seed of rank $n$ in an exchange graph, then define $N^0(u)$ to be
	the set of $\binom{n}{2}$ numbers given by the length of all geodesic loops
	distance $0$ from $u$. Similarly define $N^1(u)$ to be the set of $n
	\binom{n-1}{2}$ numbers given by the lengths of all geodesic loops distance
	$1$ from $u$.
\end{definition}

\begin{remark}
	An exchange graph automorphism $\phi \in \operatorname{Aut}\mathcal{E}$ induces an automorphism
	$\phi^\Delta \in \operatorname{Aut}\Delta$ and in this way $\phi$ induces a map $\phi_v$ which
	takes cluster variables in a seed $u$ to variables in $\phi(u)$.
\end{remark}

A geodesic loop $\mathcal{L}_u^{a,b}$ in an exchange graph $\mathcal{E}$ must get mapped to
another geodesic loop of the same length by any exchange graph automorphism,
however it is not clear that the image of $\mathcal{L}_u^{a,b}$ will be generated by
the cofrozenisation $\phi(u) \backslash \lbrace \phi_v(a),\phi_v(b) \rbrace$ rather
than another cofrozenisation with two different unfrozen vertices in $u$. The
following Lemma explains that this must always be the case.

\begin{lemma}\label{lem:fixfrozexch}
	Let $u$ be a seed in a cluster algebra $\mathcal{A}$ and $\phi \in \operatorname{Aut}\mathcal{E}_\mathcal{A}$.
	For any two vertices $a$ and $b$ the geodesic loop $\mathcal{L}^{a,b}_u$ is
	isomorphic to its image $\mathcal{L}^{\phi_v(a),\phi_v(b)}_{\phi(u)}$.
\end{lemma}

\begin{proof}
	Choose some ordering on $u$ so that the vertices $a = v_i$ and $b = v_j$ are
	indexed by $i$ and $j$ respectively, then the length of the geodesic loop
	specifies a relation $u = u \cdot \mu_i \mu_j \mu_i \dotsb$.

	For example if the loop has length $6$, then $u = u \cdot \left( \mu_i \mu_j
	\right)^3$, whereas if the length is $5$ then $u = u \cdot \mu_i \mu_j \mu_i
	\mu_j \mu_i$.

	The exchange graph automorphism $\phi$ corresponds to some $\phi_{M_n} \in
	\operatorname{Aut}_{M_n}$ which commutes with the action of $M_n$. Hence
	\[ \phi_{M_n}(u) = \phi_{M_n}(u \cdot \mu_i \mu_j \dotsb) =
	\phi_{M_n}(u) \cdot \mu_i \mu_j \dotsb \]%
	so the geodesic loop $\mathcal{L}^{\phi_v(a),\phi_v(b)}_{\phi(u)}$ has the same
	length as the geodesic loop $\mathcal{L}^{a,b}_u$, and hence the two loops are
	isomorphic.
\end{proof}

Exchange graph automorphisms preserve the combinatorial structure
around a seed. As these automorphisms are compatible with mutations the above
result could be extended to the exchange graphs of cofrozenisations with any
number of unfrozen vertices.

\begin{lemma}\label{lem:Npreserved}
	If $\phi \in \operatorname{Aut}\mathcal{E}$ is an exchange graph automorphism, with $u$ a seed and
	$v=\phi(u)$ its image, then $N^0(u) = N^0(v)$ and $N^1(u) = N^1(v)$.
\end{lemma}

\begin{lemma}\label{lem:NdefinesB}
	Given a mutation-finite diagram with at least 3 vertices, the exchange graph
	of a frozenisation leaving just two vertices unfrozen determines the weight on
	the arrow between the two unfrozen vertices.
\end{lemma}
\begin{proof}
	The exchange graph of the frozenisation leaving just two vertices $a$ and $b$
	unfrozen is a geodesic loop $\mathcal{L}^{a,b}$ with length $\operatorname{Len}\left( \mathcal{L} \right)\in
	\lbrace 4,5,6,8,\infty\rbrace$.

	If $\operatorname{Len}\left( \mathcal{L} \right)=4$ then the vertices have no arrow between them, while if
	$\operatorname{Len}\left( \mathcal{L} \right)=5$ there is a single unweighted arrow. If $\operatorname{Len}\left( \mathcal{L} \right)=6$
	then there is an arrow weighted $2$ and $\operatorname{Len}\left( \mathcal{L} \right)=8$ shows there is an
	arrow weighted $3$.

	The highest edge weight in a mutation-finite diagram (with more than 2
	vertices) is $4$, so $\operatorname{Len}\left( \mathcal{L} \right)=\infty$ implies that there is an arrow
	weighted $4$.
\end{proof}

\begin{remark}
	For any $2$-vertex diagram $B$, an edge weight of $4$ or more will always give
	$\operatorname{Len}\left( \mathcal{L} \right)=\infty$, so the exchange graph cannot determine this weight.
	However the only diagrams mutation-equivalent to $B$ are $B$ and $B^{\textrm{op}}$, so
	all diagrams in the same mutation class have the same edge weight.
\end{remark}

\subsection{Exchange graph automorphism effects on diagrams and matrices}

\begin{lemma}\label{lem:unoriented}
	An exchange graph automorphism $\phi \in \operatorname{Aut}\mathcal{E}$ takes a seed $u = (x,B)$
	to another seed $v = \phi(u) = (x',B')$ where the unoriented diagram of $B'$
	is the same as the unoriented diagram of $B$.
\end{lemma}
\begin{proof}
	Fix any two vertices $u_0$ and $u_1$ in $u$. Under $\phi_v$ these vertices are
	mapped to corresponding vertices $\phi_v(u_0) = v_0$ and $\phi_v(u_1) = v_1$
	in $v$.

	The weight on (or absence of) the arrow between $u_0$ and $u_1$ determines the
	exchange graph $\mathcal{E}_u$ of the cofrozenisation $u \backslash \lbrace u_0, u_1
	\rbrace$. By Lemma~\ref{lem:fixfrozexch}, $\mathcal{E}_u$ is isomorphic to the
	exchange graph $\mathcal{E}_v$ of the cofrozenisation $v \backslash \lbrace v_0,v_1
	\rbrace$. Hence this exchange graph determines the arrow between $v_0$ and
	$v_1$ by Lemmas~\ref{lem:Npreserved} and~\ref{lem:NdefinesB}, which
	necessarily must be the same as that between $u_0$ and $u_1$.
\end{proof}

This shows that the unoriented diagrams of two seeds related by an exchange
graph automorphism must be the same. To see how exchange graphs automorphisms
affect the orientations of the arrows we need to consider frozenisations with
three unfrozen vertices.

\begin{lemma}\label{lem:3vert-N}
	For any seed $u = (x,B)$ with 3 vertices in an exchange graph of a
	mutation-finite skew-symmetrizable diagram, the diagram of $B$ is determined
	by the sets $N^0(u)$ and $N^1(u)$, up to reversing all arrows.
\end{lemma}
\begin{proof}
	The unoriented diagram of $B$ is determined by $N^0(u) = \lbrace n_i \rbrace$,
	where each $n_i \in \lbrace 4,5,6,8,\infty \rbrace$ determines a weighted
	arrow, or absence of arrow, between two vertices.

	The orientation of $B$ (up to reversing all arrows) is given by $N^1(u)$ as
	shown in Tables~\ref{tab:diconnN},~\ref{tab:symmetricN}
	and~\ref{tab:symmetrizN}, where all mutation-finite 3-vertex diagrams are
	illustrated along with their defining sets $N^0$ and $N^1$. Hence the pair
	$\left( N^0, N^1 \right)$ defines a unique diagram, up to reversing all
	arrows.
\end{proof}

\begin{table}[p]
	\centering%
	\includegraphics{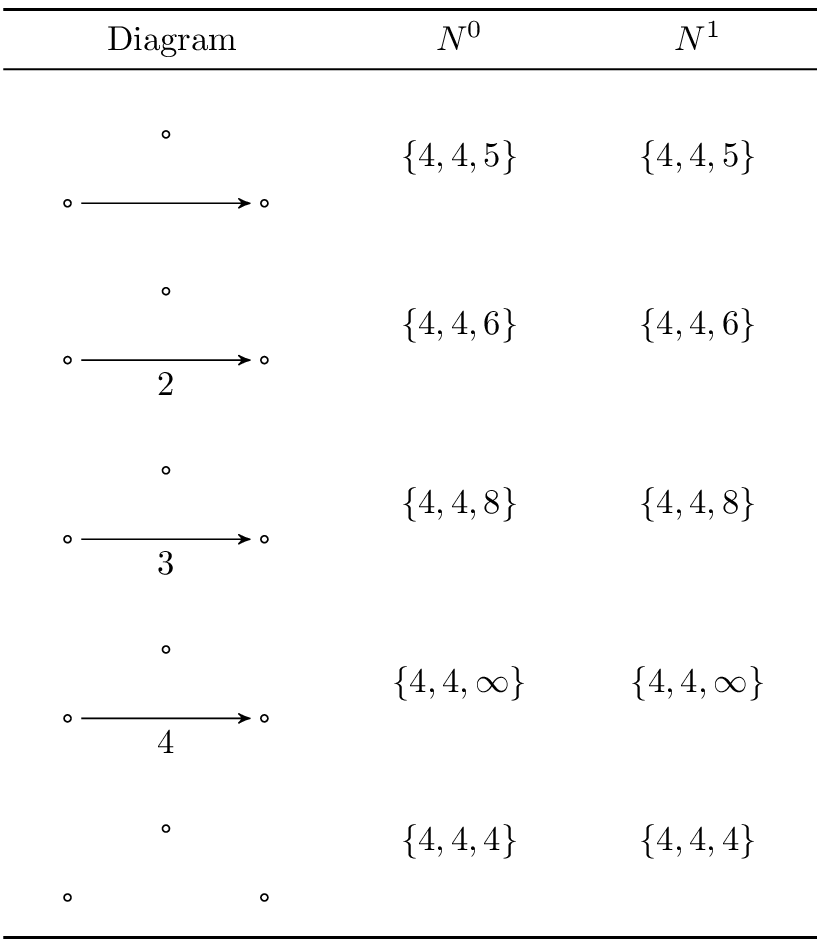}
	\caption{Disconnected 3-vertex diagrams determined by values of $N^0$.}
	\label{tab:diconnN}
\end{table}

\begin{table}[p]
	\centering%
	\includegraphics{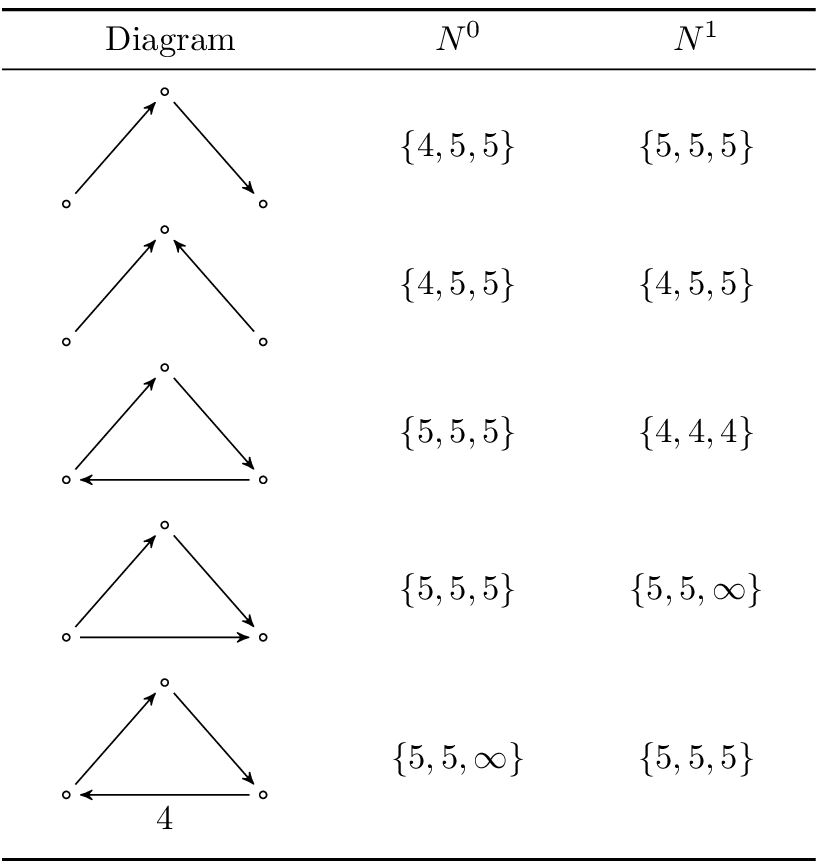}
	\caption{Connected skew-symmetric 3-vertex diagrams determined by values of
	$N^0$ and $N^1$.}
	\label{tab:symmetricN}
\end{table}

\begin{table}[p]
	\centering%
	\includegraphics{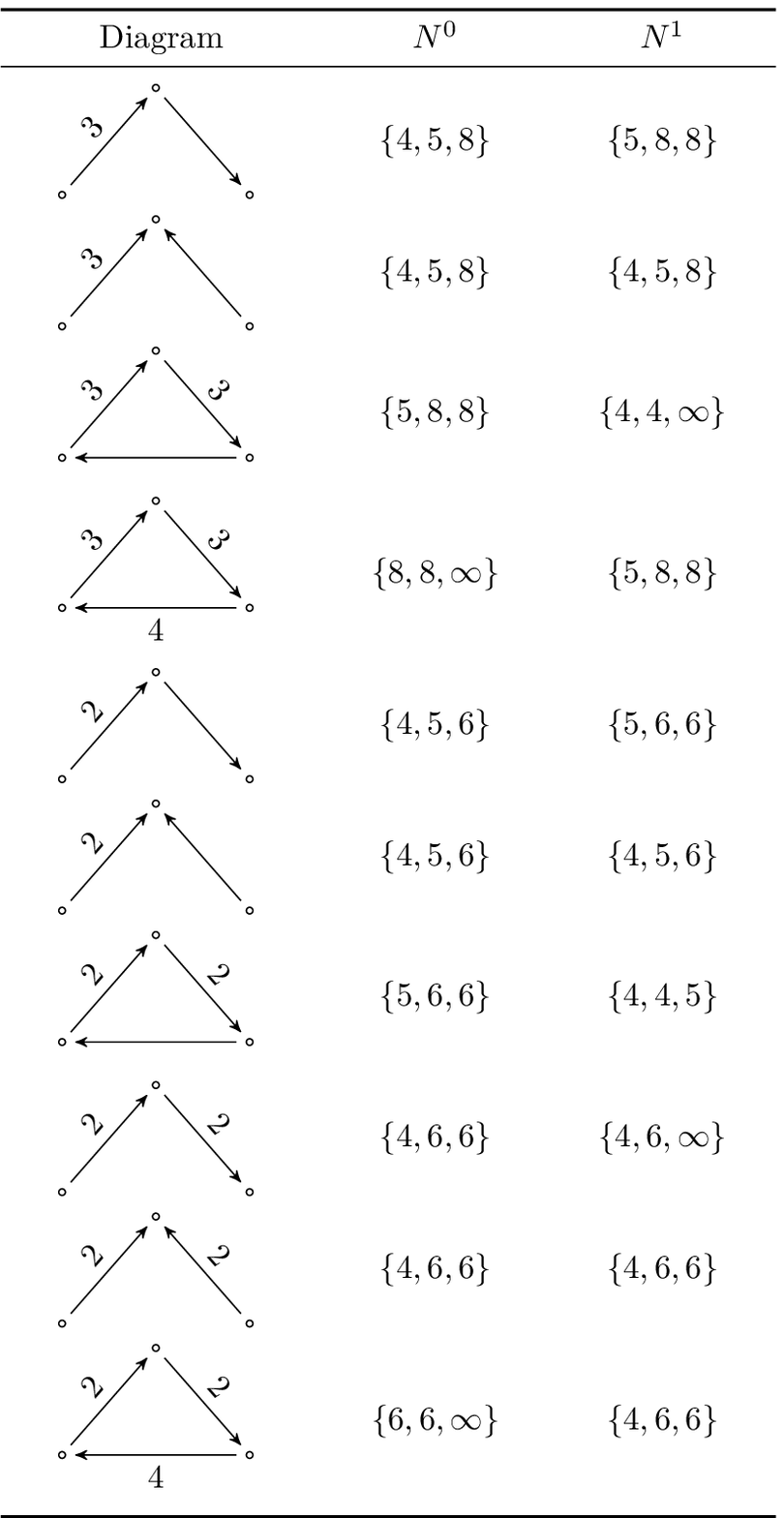}
	\caption{Connected skew-symmetrizable 3-vertex diagrams determined by values of
	$N^0$ and $N^1$.}
	\label{tab:symmetrizN}
\end{table}

In the case $N^0(u) = \lbrace 4,4,\infty \rbrace$ the diagram is of the form:
\begin{center}\includegraphics{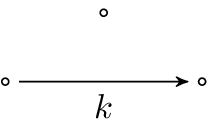}\end{center}%
where the weight satisfies $k \geq 4$ and so the diagram is not uniquely
determined.  However if $k > 4$ then the resulting diagram will never appear as
a subdiagram of any larger mutation-finite diagram. This is precisely the setup
used in the proofs below and so $N^0(u) = \lbrace 4,4,\infty \rbrace$ is always
assumed to correspond to a diagram of the form:
\begin{center}\includegraphics{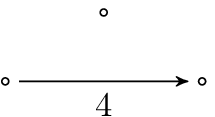}\end{center}%

\begin{proposition}\label{prop:exchdiagram}
	Let $\phi \in \operatorname{Aut}\mathcal{E}$ be an exchange graph automorphism and $u = (\textbf{x},B)$ a
	seed where $B$ is a mutation-finite skew-symmetrizable matrix with
	corresponding connected diagram $R$. In the image $\phi(u) = (\textbf{x}',B')$, the
	diagram $R'$ corresponding to the matrix $B'$ is either $R$ or $R^{\textrm{op}}$.
\end{proposition}
\begin{proof}
	Choose any 3 vertices $a,b,c$ in $u$, then by Lemma~\ref{lem:fixfrozexch}
	there is an isomorphism $\mathcal{E}\left( u \backslash \lbrace a,b,c \rbrace \right)
	\cong \mathcal{E} \left( \phi(u) \backslash \lbrace \phi(a), \phi(b),\phi(c) \rbrace
	\right)$ and $N^0(u) = N^0(\phi(u))$, $N^1(u) = N^1(\phi(u))$. Therefore by
	Lemma~\ref{lem:3vert-N} the subdiagram $S$ of $R$ consisting just of the
	arrows between $a,b$ and $c$ is the same as the subdiagram $S'$ of $R'$
	consisting of the arrows between $\phi(a),\phi(b)$ and $\phi(c)$, up to
	reversing all arrows.

	Choose a fourth vertex $d$ and consider the 3-vertex subdiagram $S_a$ on the
	vertex set $\lbrace b,c,d \rbrace$. By the same reasoning as above the image
	$S'_a = \phi(S_a)$ must be the same, but possibly with all arrows reversed.
	However both $S'$ and $S_a'$ share the edge between vertices $\phi(b)$ and
	$\phi(c)$, so if $S' = S^{\textrm{op}}$ then $S'_a = S_a^{\textrm{op}}$ whereas if $S' = S$ then
	$S'_a=S_a$.

	As $R$ is connected, by successively choosing different vertices, the whole
	diagram $R'$ must either be the same as $R$ or $R^{\textrm{op}}$.
\end{proof}

This shows that any exchange graph automorphism takes clusters to clusters and
a diagram to itself or its opposite. However this is not enough to show that
these automorphisms are cluster automorphisms, as this requires the matrix $B$ of
the diagram to be sent to $\pm B$. For this we require the markings on the
exchange graph.

\begin{proposition}\label{prop:mexchmatrix}
	Given a marked exchange graph automorphism $\phi \in \operatorname{Aut}\widehat{\mathcal{E}}$ and a seed $u
	= (x,B)$ with image $\phi(u) = (x',B')$, then the matrix $B' = B$ or $-B$.
\end{proposition}
\begin{proof}
	Let $R$ be the diagram associated to $B$, and let $R'$ be the diagram
	associated to $B'$. Let $D_B$ be the symmetrizing matrix for $B$, each vertex
	$v_k$ in $u$ has a symmetrizing multiplier, which marked exchange graph
	automorphisms preserve, so each vertex $\phi_v(v_k)$ in $\phi(u)$ has the same
	symmetrizing multiplier as $v_k$ and $D_B = D_{B'}$. 
	
	By Proposition~\ref{prop:exchdiagram}, $R'$ is the same as $R$ or $R^{\textrm{op}}$
	with symmetrizing matrix $D_{B'} = D_B$ which defines the skew-symmetrizable
	matrix $B' = B$ or $-B$.
\end{proof}

These results ensure that a marked exchange graph automorphism fixes matrices in
seeds and so correspond to cluster automorphisms. In this way we generalise
Corollary~\ref{cor:mfquiv-exchaut} to all mutation-finite skew-symmetrizable
matrices.

\begin{theorem}\label{thm:mexchaut}
	For a seed $(\textbf{x},B)$ where $B$ is a mutation-finite skew-symmetrizable matrix
	with mutation class $\mathcal{S}$, cluster algebra $\mathcal{A} = \mathcal{A}(\mathcal{S})$ and marked
	exchange graph $\widehat{\mathcal{E}}_\mathcal{A}$ then
	\[ \operatorname{Aut} \mathcal{A} = \operatorname{Aut} \widehat{\mathcal{E}}_\mathcal{A}. \]%
\end{theorem}

\begin{proof}
	A cluster automorphism $f \in \operatorname{Aut}\mathcal{A}$ satisfies the following properties:
	\begin{itemize}
		\item $f(\textbf{x})$ is a cluster
		\item $f$ is compatible with mutations
		\item $B(f(\textbf{x})) \cong B$ or $-B$
	\end{itemize}
	for all seeds $(x,B)$ in the mutation class $\mathcal{S}$. Such an automorphism induces an
	automorphism of the exchange graph, and as $f$ sends a matrix $B$ to $\pm B$
	it also fixes the symmetrizing matrix so fixes the marking on the exchange
	graph. Therefore $f \in \operatorname{Aut}\widehat{\mathcal{E}}_\mathcal{A}$ and $\operatorname{Aut}\mathcal{A} \subset \operatorname{Aut}\widehat{\mathcal{E}}_\mathcal{A}$.

	To show that $\operatorname{Aut}\widehat{\mathcal{E}}_\mathcal{A} \subset \operatorname{Aut}\mathcal{A}$ let $(\textbf{x},B)$ be a labelled
	seed, with mutation class $\mathcal{S}^0$ and quotient $\mathcal{S}$. If $\phi \in
	\operatorname{Aut}\widehat{\mathcal{E}}(\mathcal{S}) \subset \operatorname{Aut}\mathcal{E}(\mathcal{S})$, then by Theorem~\ref{thm:pullback}
	this pulls back to an automorphism $\phi^\Delta \in \operatorname{Aut}\Delta(\mathcal{S}^0)$. Then
	the image $\phi^\Delta(\textbf{x}) = (y_1,\dotsc,y_n)$ where $\textbf{x} = (x_1,\dotsc,x_n)$
	gives an automorphism $f: \mathbb{C}(x_1,\dotsc,x_n) \to
	\mathbb{C}(x_1,\dotsc,x_n)$ defined by $f(x_i) = y_i$.

	This $f$ then corresponds to $\phi$, so $f(\textbf{x})$ is a cluster and it remains
	to show that $B(f(\textbf{x})) = \pm B = \pm B(\textbf{x})$, however this follows from
	Proposition~\ref{prop:mexchmatrix} so $f\in \operatorname{Aut}\mathcal{A}$.
\end{proof}

\section{Unfoldings}\label{sec:unf}

Many skew-symmetrizable matrices $B$ have unfoldings to skew-symmetric matrices
$C$, which extend to seeds, where a given seed in $\mathcal{S}(B)$ unfolds to a seed in
$\mathcal{S}(C)$. The corresponding exchange graphs are related, with the marked
exchange graph $\widehat{\mathcal{E}}(B)$ embedding into the exchange graph $\mathcal{E}(C)$ provided
edges marked in certain ways split into multiple edges.

\begin{definition}[{\cite[Section 4]{FST-unfoldings}}]%
	Given a skew-symmetrizable $n \times n$ matrix $B = \left( b_{i,j} \right)$
	with symmetrizing matrix $D = \operatorname{diag}\left( d_i \right)$, let $ m = \sum_{j = 1}^{n} d_j$
	and partition the set $\lbrace 1,\dotsc, m \rbrace$ into $n$ disjoint
	consecutive index sets $E_i$ such that $\left| E_j \right| = d_j$ for all $j$.

	Construct a skew-symmetric $m \times m$ matrix $C$ where:
	\begin{enumerate}

		\item\label{enum:unfdefa} The sum of entries in each column of each $E_i
			\times E_j$ block equals $b_{i,j}$.

		\item\label{enum:unfdefb} If $b_{i,j} > 0$ then all entries in the $E_i
			\times E_j$ block are non-negative.

		\item\label{enum:unfdefc} All entries in each $E_i \times E_i$ block are
			zero.

	\end{enumerate}
	Given $i \in \lbrace 1, \dotsc,n \rbrace$ and any $j,k \in E_i$ the
	corresponding mutations $\mu_j$ and $\mu_k$ commute.  The $i$-th
	\textit{composite mutation} $\widetilde{\mu}_i$ of $C$ is given by
	\[ \widetilde{\mu_i} = \prod_{j \in E_i} \mu_j. \]%

	The matrix $C$ is the \textit{unfolding} of $B$ if the matrix $C' = C \cdot
	\left(\widetilde{\mu}_{k_1} \widetilde{\mu}_{k_2} \dotsb \widetilde{\mu}_{k_r}\right)$
	satisfies the conditions~\ref{enum:unfdefa} and~\ref{enum:unfdefb} above with
	respect to the matrix $B' = B \cdot \left( \mu_{k_1} \mu_{k_2} \dotsb
	\mu_{k_r} \right)$ for any sequence of mutations $\mu_{k_i}$ with corresponding
	composite mutations $\widetilde{\mu}_{k_i}$.
\end{definition}

A labelled seed $\left( [\beta_i], B \right)$, with skew-symmetrizable matrix
$B$, unfolds in the same way to $\left( [\gamma_i], C \right)$ where $C$ is the
unfolding of $B$. The $j$-th row in $B$ corresponds to the cluster variable
$\beta_j$ and this row unfolds to $d_j$ rows in $C$, hence $\beta_j$ unfolds to
$d_j$ cluster variables $\lbrace \gamma_{j_1}, \dotsc, \gamma_{j_{d_j}}
\rbrace$.

\begin{remark}
	A diagram has a finite number of distinct matrix representations, each of
	which may give different unfoldings, or may not admit any unfolding. Almost
	all mutation-finite matrices have an unfolding.
\end{remark}

\begin{definition}
	Given a permutation $\sigma \in \operatorname{Sym}(n)$ of the initial seed, construct the
	\textit{composite permutation} $\widetilde{\sigma} \in \operatorname{Sym}(m)$ to be the
	permutation given by:
	\begin{center}\includegraphics{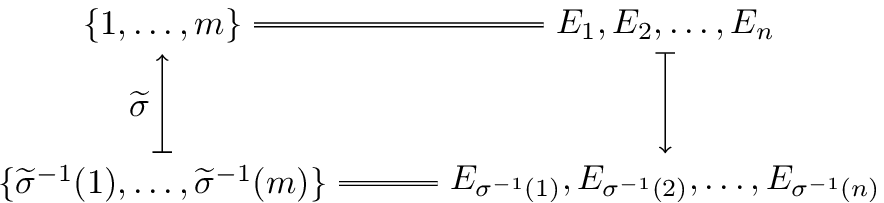}\end{center}
\end{definition}

\begin{theorem}\label{thm:autsintounf}
	Given a skew-symmetrizable matrix $B$ which unfolds to a matrix $C$, with
	corresponding marked exchange graphs $\widehat{\mathcal{E}}(B)$ and $\mathcal{E}(C) = \widehat{\mathcal{E}}(C)$,
	then
	\[ \operatorname{Aut} \widehat{\mathcal{E}}(B) \hookrightarrow \operatorname{Aut} \mathcal{E}(C). \]%
\end{theorem}

\begin{proof}
	Choose an initial $n \times n$ labelled seed $u = \left( [\beta_i], B \right)$
	which unfolds to the $m \times m$ labelled seed
	$\left( [\gamma_i], C \right)$ with index sets $E_k$ for $k = 1, \dotsc,n$ and
	$\beta_i \leadsto \lbrace \gamma_j \rbrace_{j \in E_i}$.

	Let $\phi \in \operatorname{Aut} \widehat{\mathcal{E}}(B)$ be an exchange graph automorphism, then $\phi$
	corresponds to both a cluster automorphism $f \in \operatorname{Aut} \mathcal{A}_B$ of the cluster
	algebra $\mathcal{A}_B$, constructed from the initial seed $u$, and to a mutation class
	automorphism $\phi_M \in \operatorname{Aut}_{M_n}\mathcal{S}^0(B)$. This mutation class automorphism
	in turn corresponds to an element of $M_n$, so there is a sequence of $r$
	mutations $\mu_{k_i}$ and a permutation $\sigma$ such that
	\[ \phi_M(u) = u \cdot \left( \mu_{k_1} \mu_{k_2} \dotsb \mu_{k_r} \sigma
	\right). \]%
	All such automorphisms are constructed to have the same action on the initial
	seed $u$, so
	\[ \phi(u) = \left( \left[ \widebar{\beta}_i \right], \pm B \right) = \left(
	\left[ f(\beta_i) \right], \pm B \right) = \phi_M(u) = u \cdot \left(
\mu_{k_1} \mu_{k_2} \dotsb \mu_{k_r} \sigma \right). \]%

	In the unfolding, each mutation $\mu_{k_i}$ corresponds to the composite
	mutation $\widetilde{\mu}_{k_i}$ and the permutation $\sigma$ corresponds to
	the composite permutation $\widetilde{\sigma}$, so the following commutes:
	\begin{center}\includegraphics{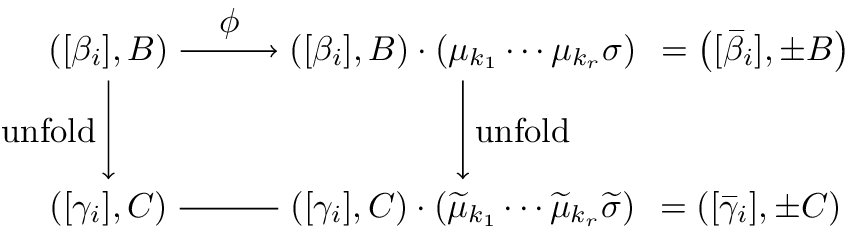}\end{center}
	The automorphism $\phi$ corresponds to a cluster automorphism, so the matrix
	of the image of $u$ is $\pm B$. The seed $\phi(u) = \left(
	\left[\widebar{\beta}_i \right],\pm B \right)$ unfolds to $\left( \left[
	\widebar{\gamma}_i \right], \pm C \right)$ and hence $\left(
	\widetilde{\mu}_{k_1} \dotsb \widetilde{\mu}_{k_r} \widetilde{\sigma} \right)
	\in M_m$ acts on $\left( [\gamma_i], C \right)$ to give a seed with the same
	matrix up to sign, so corresponds to a cluster automorphism of the cluster
	algebra constructed with $\left( [\gamma_i], C \right)$ as the initial seed,
	and hence to an automorphism of the exchange graph $\mathcal{E}(C)$.
\end{proof}

\begin{corollary}
	By Theorem~\ref{thm:mexchaut} the marked exchange graph automorphisms
	correspond to cluster automorphisms, so for a skew-symmetrizable matrix $B$
	which unfolds to $C$ and with corresponding cluster algebras $\mathcal{A}_B$ and
	$\mathcal{A}_C$, Theorem~\ref{thm:autsintounf} implies
	\[ \operatorname{Aut}\mathcal{A}_B \hookrightarrow \operatorname{Aut}\mathcal{A}_C. \]%
\end{corollary}

\begin{figure}
	\centering%
	\includegraphics{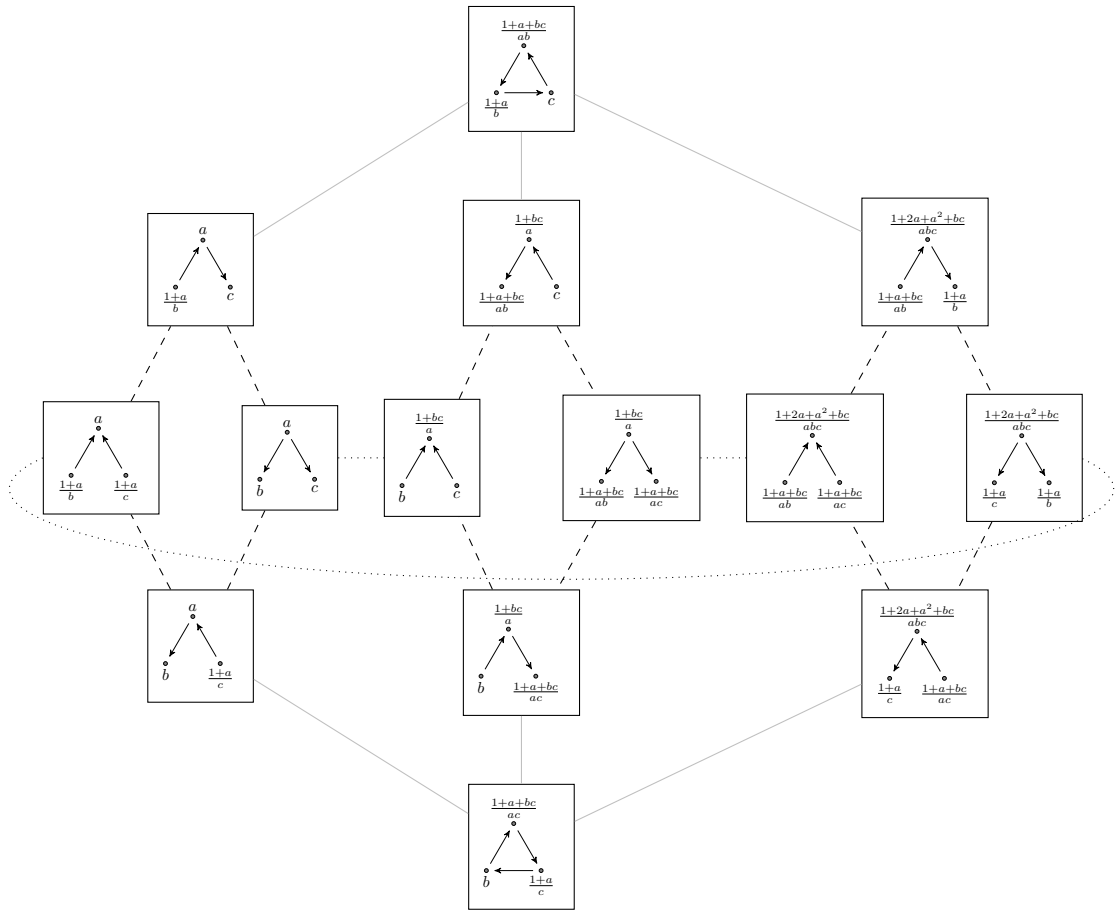}
	\caption{Exchange graph of the mutation class of type $A_3$. The dotted and
		dashed edges show how the marked exchange graph of type $B_2$ shown in
		Figure~\ref{fig:b2-marked} unfolds. A dashed edge in
		Figure~\ref{fig:b2-marked} corresponds to the composite mutation denoted by a
		consecutive pair of dashed edges in this figure.}
	\label{fig:a3-exch}
\end{figure}

\begin{example}
	The matrix $B$ representing the Dynkin diagram of type $B_2$
	\[ B = \begin{pmatrix} 0 & 1 \\ -2 & 0 \end{pmatrix} \quad \textrm{unfolds to}
	\quad C = \begin{pmatrix} 0 & 1 & 1 \\ -1 & 0 & 0 \\ -1 & 0 & 0 \end{pmatrix},
	\]%
	the matrix representing a quiver of Dynkin type $A_3$. The symmetrizing matrix
	of $B$ is given by $D = \operatorname{diag}\left( 1,2 \right)$ so the $B_2$ marked exchange graph shown
	in Figure~\ref{fig:b2-marked} embeds into the exchange graph of type $A_3$
	shown in Figure~\ref{fig:a3-exch}. The dashed edges in
	Figure~\ref{fig:b2-marked} correspond to the pairs of dashed edges
	representing composite mutations in Figure~\ref{fig:a3-exch}. Dotted edges in
	Figure~\ref{fig:b2-marked} correspond to single dotted edges in
	Figure~\ref{fig:a3-exch}.

	The seed $\left( [x,y], B \right)$ unfolds to the seed $\left( [a,b,c], C
	\right)$ and the cluster variables of these two seeds are related with
	\[ x \leadsto a, \quad y \leadsto \lbrace b,c \rbrace \]%
	as the symmetrizing matrix $\operatorname{diag}\left( 1,2 \right)$ ensures that $y$ unfolds to two
	cluster variables.

	The automorphism $\phi \in \operatorname{Aut}\widehat{\mathcal{E}}(B_2)$ given by rotation by
	$\frac{2\pi}{3}$ takes the seed $[x,y]$ to
	$\left[\frac{1+y^2}{x},\frac{1+x+y^2}{xy}\right]$ and corresponds to the
	cluster automorphism $f \in \operatorname{Aut}\mathcal{A}(B_2)$ given by
	\[ f(x) = \frac{1+y^2}{x}, \quad f(y)	= \frac{1+x+y^2}{xy}. \]%

	This automorphism induces an automorphism of the exchange graph of $A_3$ given
	by a rotation along the embedded $\widehat{\mathcal{E}}(B_2)$ fixing the seeds with cyclic
	quivers and takes $\left[ a,b,c \right]$ to $\left[ \frac{1+bc}{a},
	\frac{1+a+bc}{ab}, \frac{1+a+bc}{ac} \right]$ which corresponds to the cluster
	automorphism $g\in \operatorname{Aut}\mathcal{A}(A_3)$ given by
	\[ g(a) = \frac{1+bc}{a}, \quad g(b) = \frac{1+a+bc}{ab}, \quad g(c) =
	\frac{1+a+bc}{ac}. \]%

	However the automorphism could also correspond to the cluster automorphism
	$\tilde{g}\in \operatorname{Aut}\mathcal{A}(A_3)$ where
	\[ \tilde{g}(a) = g(a), \quad \tilde{g}(b) =	\frac{1+a+bc}{ac}, \quad
	\tilde{g}(c) = \frac{1+a+bc}{ab}. \]%

	There is a single non-identity $\mathcal{E}(A_3)$ exchange graph automorphism which
	fixes the embedded $\widehat{\mathcal{E}}(B_2)$, given by a reflection in the circle of the
	embedded subgraph and interchanging the two seeds with cyclic quivers. This
	then corresponds to the cluster automorphism $h \in \operatorname{Aut}\mathcal{A}(A_3)$ given by
	\[ h(a) = a, \quad h(b) = c, \quad h(c) = b \]%
	such that $\tilde{g} = g \circ h = h \circ g$.

\end{example}

Theorem~\ref{thm:autsintounf} shows that cluster automorphisms of $\mathcal{A}_B$
commute with unfolding the seeds, so a direct cluster automorphism
$\phi\in\operatorname{Aut}\mathcal{A}_B$ preserves the exchange matrix $B$, which when unfolded to
$\psi\in\operatorname{Aut}\mathcal{A}_C$ must also preserve the exchange matrix $C$ and so
is also a direct cluster automorphism.

\begin{corollary}\label{cor:direct_subgroup}
	$\operatorname{Aut}^+\mathcal{A}_B \hookrightarrow \operatorname{Aut}^+\mathcal{A}_C$.
\end{corollary}

\section{Mapping class groups}\label{sec:mcg}

In their paper introducing cluster automorphisms~\cite{ASS-auts} Assem,
Schiffler and Shramchenko introduced the tagged mapping class group
for surfaces with punctures. This group has been shown to coincide with the
group of direct cluster automorphisms of the surface's corresponding cluster
algebra.

\begin{definition}
	Given a surface with marked points $(S,M)$ the \textit{mapping class group} of
	the surface is given by
	\[ \operatorname{MCG}(S,M) = \mathchoice{\left.\raisebox{0.5ex}{$\operatorname{Homeo}^+(S,M)$}\middle/\raisebox{-1ex}{$\operatorname{Homeo}^0(S,M)$}\right.}{\left.\operatorname{Homeo}^+(S,M)\middle/\operatorname{Homeo}^0(S,M)\right.}{\operatorname{Homeo}^+(S,M)/\operatorname{Homeo}^+(S,M)}{\operatorname{Homeo}^+(S,M)/\operatorname{Homeo}^0(S,M)}. \]%
	Here $\operatorname{Homeo}^+(S,M)$ is the group of orientation-preserving homeomorphisms
	from $S$ to itself which sends the set $M$ to itself, but does not necessarily
	fix $M$ nor the boundary of $S$ pointwise, and $\operatorname{Homeo}^0(S,M)$ is the subgroup of
	homeomorphisms which are isotopic to the identity such that the isotopy fixes
	$M$ pointwise.
\end{definition}

The cluster structure given by triangulations of a surface with marked points
was first studied by Fomin, Shapiro and Thurston in~\cite{FST-tri}, where they
show that flips of arcs in a triangulation coincide with mutations. However such
a triangulation could contain self-folded triangles, and therefore arcs that
cannot be flipped; to get around this problem, the authors introduced taggings
on the arcs. A \textit{tagged arc} is an arc which does not cut out a
once-punctured monogon, where the enpoints are tagged either \textit{plain} or
\textit{notched}, such that any endpoints on $\partial S$ are tagged plain and if
the endpoints of an arc coincide then they must be tagged the same.

Two tagged arcs are \textit{compatible} if either their underlying arcs are the
same and then at least one endpoint must be tagged in the same way, or the
underlying arcs are not equal but are compatible. In this case, if they share an
endpoint, the arcs must be tagged in the same way at that endpoint. A tagged
triangulation is a maximal collection of compatible tagged arcs and a tagged
flip is then defined in the same way as for triangulations, where a tagged arc
is replaced with the unique other compatible tagged arc and these flips again
correspond to mutations. See~\cite[Section 7]{FST-tri} or~\cite[Section
4]{ASS-auts} for more details.

\begin{definition}
	The \textit{tagged mapping class group} of a surface $(S,M)$ with $p$
	punctures is the semidirect product of the standard mapping class group of the
	surface with $\mathbb{Z}_2^p$,
	\[ \operatorname{MCG}_{\bowtie}(S,M) = \mathbb{Z}_2^p \rtimes \operatorname{MCG}(S,M), \]%
	where the elements of $\operatorname{MCG}(S,M)$ act as diffeomorphisms on the surface and
	elements of $\mathbb{Z}_2^p$ switch or preserve the tags on the tagged
	triangulation at each puncture.
\end{definition}

\begin{theorem}[{\cite[Theorem 4.11]{ASS-auts}}]\label{thm:mcginj}
	Let $(S,M)$ be a surface with $p$ punctures, with corresponding cluster
	algebra $\mathcal{A}$, then
	\begin{enumerate}
		\item $\operatorname{MCG}(S)$ is isomorphic to a subgroup of $\operatorname{Aut}^+\mathcal{A}$.
		\item If $p \geq 2$ or $\partial S \not = \emptyset$ then
			$\operatorname{MCG}_{\bowtie}(S)$ is isomorphic to a subgroup of $\operatorname{Aut}^+\mathcal{A}$.
	\end{enumerate}
\end{theorem}

They showed that for discs and annuli without punctures as well as for certain
discs with 1 or 2 punctures then the tagged mapping class group is isomorphic to the
group of direct cluster automorphisms of the corresponding cluster algebra. The
authors conjectured that this would be the case for almost all surfaces with
marked points.
Br\"{u}stle and Qiu proved that this conjecture is true in~\cite{BQ-mcg}:

\begin{theorem}[{\cite[Theorem 4.7]{BQ-mcg}}]\label{thm:mcgiso}
	Let $(S,M)$ be a surface with marked points which is not
	\begin{enumerate}
			\item a once-punctured disc with 2 or 4 marked points on the boundary
			\item a twice-punctured disc with 2 marked points on the boundary
	\end{enumerate}
	then
	\[ \operatorname{MCG}_{\bowtie}(S,M) = \operatorname{Aut}^+\mathcal{A}. \]%
\end{theorem}

Theorem~\ref{thm:mcginj} shows that $\operatorname{MCG}_{\bowtie}(S,M) \hookrightarrow \operatorname{Aut}^+\mathcal{A}$, so
the proof of Theorem~\ref{thm:mcgiso} needs to show that this injection is
surjective. This follows from the result below proved by Bridgeland and Smith:

\begin{proposition}[{\cite[Prop. 8.5]{BS-tri}}]%
	Suppose $(S,M)$ is a surface which is not one of:
	\begin{enumerate}%
		\item a sphere with $\leq 5$ marked points;
		\item an unpunctured disc with $\leq 3$ marked points on the boundary;
		\item a disc with a single puncture and one marked point on the boundary;
		\item a once-punctured disc with 2 or 4 marked points on the boundary;
		\item a twice-punctured disc with 2 marked points on the boundary,
	\end{enumerate}%
	then two tagged triangulations of $(S,M)$ differ by an element of
	$\operatorname{MCG}_{\bowtie}(S,M)$ if and only if the associated quivers are isomorphic.
\end{proposition}

\subsection{Unfoldings and covering maps}

Diagrams correspond to triangulations of orbifolds in the same way that quivers
correspond to triangulations of surfaces. A covering of the orbifold by a surface
corresponds to an unfolding of the diagram to a quiver, in such a way that
composite mutations of the quiver correspond to triangle flips in the
triangulation of the surface, as discussed in~\cite{FST-orbifolds}.

In their paper on the growth rate of cluster algebras, Felikson, Shapiro, Thomas
and Tumarkin~\cite{FSTT-growth} defined the mapping class group of a cluster
algebra $\operatorname{MCG}(\mathcal{A})$ to be the elements of $M_n$ which fix the initial exchange
matrix up to a quotient by those elements of $M_n$ which fix the initial seed.
Elements of this group would then fix the initial exchange matrix and map the
initial cluster to some other cluster in the mutation class, and hence would
induce a direct cluster automorphism.

Fix a marked orbifold $\mathcal{O}$ with $m$ punctures. In~\cite[Remark
4.15]{FSTT-growth} the cluster mapping class group is argued to either contain
the orbifold's mapping class group as a proper normal subgroup with quotient
$\mathchoice{\left.\raisebox{0.5ex}{$\operatorname{MCG}(\mathcal{A})$}\middle/\raisebox{-1ex}{$\operatorname{MCG}(\mathcal{O})$}\right.}{\left.\operatorname{MCG}(\mathcal{A})\middle/\operatorname{MCG}(\mathcal{O})\right.}{\operatorname{MCG}(\mathcal{A})/\operatorname{MCG}(\mathcal{A})}{\operatorname{MCG}(\mathcal{A})/\operatorname{MCG}(\mathcal{O})} \cong \mathbb{Z}^{m}_{2}$ (when $m>1$, or when $m = 1$ and
the boundary non-empty) or be isomorphic to the orbifold mapping class group
(when $m=0$, or when $m=1$ and the boundary is empty).

The additional $\mathbb{Z}_2$ for each interior marked point corresponds to the
additional taggings in the definition of $\operatorname{MCG}_{\bowtie}(\mathcal{O})$ and so
suggests that the following would be true:

\begin{conj}\label{conj:orbmcg}
	For a cluster algebra $\mathcal{A}$ arising from the triangulation of an orbifold
	$\mathcal{O}$
	\[ \operatorname{MCG}_{\bowtie}(\mathcal{O}) \cong \operatorname{Aut}^+\mathcal{A}. \]%
\end{conj}

\begin{figure}
	\centering%
	\includegraphics{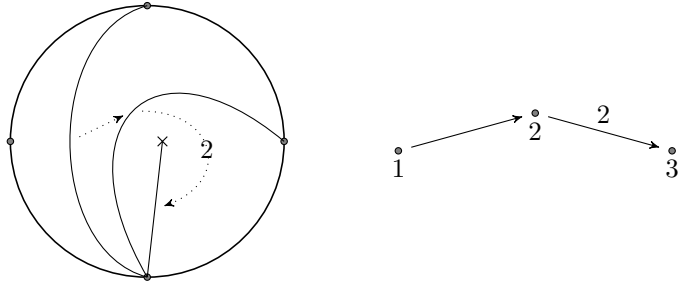}
	\caption{Triangulation of an orbifold (left) with associated diagram (right). The interior
	orbifold point is shown as a cross. The associated diagram is also shown with dotted
	arrows inside the triangulation.}
	\label{fig:mcg-b3}
\end{figure}

\begin{example}
	Consider the orbifold $\mathcal{O}$ constructed from the disc with four marked points
	on the boundary and a single orbifold point in the interior, as shown in
	Figure~\ref{fig:mcg-b3}.

	This orbifold has no punctures, so the tagged mapping class group is equal to
	the mapping class group. Any element of the mapping class group must fix the
	orbifold point and permute the four boundary marked points. The only such
	permutations are rotations around the boundary, as any reflection would not
	preserve the orientation, hence the mapping class group is isomorphic to
	$\mathbb{Z}_4$ generated by a rotation by angle $\frac{\pi}{2}$.

	This orbifold corresponds to the cluster algebra of Dynkin type $B_3$, which
	can be generated by the diagram in Figure~\ref{fig:mcg-b3}. The cluster
	automorphism group of $\mathcal{A}_{B_3}$ is the dihedral group with 8 elements:
	\[ \operatorname{Aut}\mathcal{A}_{B_3} \cong D_{4} = \mathbb{Z}_4 \rtimes \mathbb{Z}_2, \]%
	where $\mathbb{Z}_4$ is generated by the automorphism given by the action of $\mu_1
	\mu_2 \mu_3$ on the initial cluster and $\mathbb{Z}_2$ by $\mu_1 \mu_3$. This can be
	seen as the automorphisms of the marked exchange graph shown in
	Figure~\ref{fig:b3-marked} where the 4 squares are permuted while fixing the
	markings.
	
	The direct cluster automorphisms are those in the subgroup $\mathbb{Z}_4$ of the
	cluster automorphism group, and so
	\[ \operatorname{Aut}^{+}\mathcal{A}_{B_3} \cong \mathbb{Z}_4 \cong \operatorname{MCG}(\mathcal{O}) = \operatorname{MCG}_{\bowtie}(\mathcal{O}). \]%
\end{example}

\subsection*{Acknowledgements}
The author would like to thank Pavel Tumarkin for his help and
supervision, as well as Anna Felikson and Robert Marsh for highlighting the
links between this work and maximal green sequences.

\newcommand{\doi}[1]{\href{http://dx.doi.org/#1}{\texttt{doi:#1}}}
\newcommand{\arxiv}[1]{\href{http://arxiv.org/abs/#1}{\texttt{arXiv:#1}}}


\begin{thebibliography}{10}

\bibitem{ASS-auts}
Ibrahim Assem, Ralf Schiffler, and Vasilisa Shramchenko.
\newblock Cluster automorphisms.
\newblock {\em Proc. Lond. Math. Soc. (3)}, 104(6):1271--1302, 2012.
\newblock \par\arxiv{1009.0742}, \doi{10.1112/plms/pdr049}.

\bibitem{BS-tri}
Tom Bridgeland and Ivan Smith.
\newblock Quadratic differentials as stability conditions.
\newblock {\em Publ. Math. Inst. Hautes \'Etudes Sci.}, 121:155--278, 2015.
\newblock \par\arxiv{1302.7030}, \doi{10.1007/s10240-014-0066-5}.

\bibitem{BDP-green}
Thomas Br{\"u}stle, Gr{\'e}goire Dupont, and Matthieu P{\'e}rotin.
\newblock On maximal green sequences.
\newblock {\em Int. Math. Res. Not. IMRN}, (16):4547--4586, 2014.
\newblock \par\arxiv{1205.2050}, \doi{10.1093/imrn/rnt075}.

\bibitem{BQ-mcg}
Thomas Br{\"u}stle and Yu~Qiu.
\newblock Tagged mapping class groups: {A}uslander-{R}eiten translation.
\newblock {\em Math. Z.}, 279(3-4):1103--1120, 2015.
\newblock \par\arxiv{1212.0007}, \doi{10.1007/s00209-015-1405-z}.

\bibitem{CCC-bwm}
Sergio Cecotti, Clay Cordova, and Cumrun Vafa.
\newblock Braids, walls, and mirrors.
\newblock \par\arxiv{1110.2115}.

\bibitem{CZ-exch}
Wen Chang and Bin Zhu.
\newblock Cluster automorphism groups and automorphism groups of exchange
  graphs.
\newblock \par\arxiv{1506.02029}.

\bibitem{CZ-fintype}
Wen Chang and Bin Zhu.
\newblock Cluster automorphism groups of cluster algebras of finite type.
\newblock {\em J. Algebra}, 447:490--515, 2016.
\newblock \par\arxiv{1506.01950}, \doi{10.1016/j.jalgebra.2015.09.045}.

\bibitem{CZ-rooted}
Wen Chang and Bin Zhu.
\newblock On rooted cluster morphisms and cluster structures in
  2-{C}alabi-{Y}au triangulated categories.
\newblock {\em J. Algebra}, 458:387--421, 2016.
\newblock \par\arxiv{1410.5702}, \doi{10.1016/j.jalgebra.2016.03.042}.

\bibitem{FSTT-growth}
Anna Felikson, Michael Shapiro, Hugh Thomas, and Pavel Tumarkin.
\newblock Growth rate of cluster algebras.
\newblock {\em Proc. Lond. Math. Soc. (3)}, 109(3):653--675, 2014.
\newblock \par\arxiv{1203.5558}, \doi{10.1112/plms/pdu010}.

\bibitem{FST-orbifolds}
Anna Felikson, Michael Shapiro, and Pavel Tumarkin.
\newblock {C}luster algebras and triangulated orbifolds.
\newblock {\em Adv. Math.}, 231(5):2953--3002, 2012.
\newblock \par\arxiv{1111.3449}, \doi{10.1016/j.aim.2012.07.032}.

\bibitem{FST-unfoldings}
Anna Felikson, Michael Shapiro, and Pavel Tumarkin.
\newblock {C}luster algebras of finite mutation type via unfoldings.
\newblock {\em Int. Math. Res. Not. IMRN}, (8):1768--1804, 2012.
\newblock \par\arxiv{1006.4276}, \doi{10.1093/imrn/rnr072}.

\bibitem{FST-mutfin}
Anna Felikson, Michael Shapiro, and Pavel Tumarkin.
\newblock {S}kew-symmetric cluster algebras of finite mutation type.
\newblock {\em J. Eur. Math. Soc. (JEMS)}, 14(4):1135--1180, 2012.
\newblock \par\arxiv{0811.1703}, \doi{10.4171/JEMS/329}.

\bibitem{FST-tri}
Sergey Fomin, Michael Shapiro, and Dylan Thurston.
\newblock Cluster algebras and triangulated surfaces. {I}. {C}luster complexes.
\newblock {\em Acta Math.}, 201(1):83--146, 2008.
\newblock \par\arxiv{math/0608367}, \doi{10.1007/s11511-008-0030-7}.

\bibitem{FZ-CA1}
Sergey Fomin and Andrei Zelevinsky.
\newblock Cluster algebras. {I}. {F}oundations.
\newblock {\em J. Amer. Math. Soc.}, 15(2):497--529 (electronic), 2002.
\newblock \par\arxiv{math/0104151}, \doi{10.1090/S0894-0347-01-00385-X}.

\bibitem{FZ-CA2}
Sergey Fomin and Andrei Zelevinsky.
\newblock Cluster algebras. {II}. {F}inite type classification.
\newblock {\em Invent. Math.}, 154(1):63--121, 2003.
\newblock \par\arxiv{math/0208229}, \doi{10.1007/s00222-003-0302-y}.

\bibitem{GSV-exch}
Michael Gekhtman, Michael Shapiro, and Alek Vainshtein.
\newblock On the properties of the exchange graph of a cluster algebra.
\newblock {\em Math. Res. Lett.}, 15(2):321--330, 2008.
\newblock \par\arxiv{math/0703151}, \doi{10.4310/MRL.2008.v15.n2.a10}.

\bibitem{K-dilog}
Bernhard Keller.
\newblock On cluster theory and quantum dilogarithm identities.
\newblock In {\em Representations of algebras and related topics}, EMS Ser.
  Congr. Rep., pages 85--116. Eur. Math. Soc., Z\"urich, 2011.
\newblock \par\arxiv{1102.4148}, \doi{10.4171/101-1/3}.

\bibitem{KP-auts}
Alastair King and Matthew Pressland.
\newblock Labelled seeds and the mutation group.
\newblock {\em Math. Proc. Camb. Phil. Soc.}
\newblock Electronically published, to appear in print.
\newblock \par\arxiv{1309.6579}, \doi{10.1017/S0305004116000918}.

\end{thebibliography}
\end{document}